\documentclass[english, 10pt]{amsart}

\usepackage{babel}
\usepackage[utf8x]{inputenc}
\usepackage{mathrsfs}
\usepackage{amsmath}
\usepackage{amssymb}
\usepackage{amsthm}
\usepackage[all]{xy}
\usepackage{leftidx}
\usepackage{amsmath,amssymb,amsfonts}
\usepackage{amsmath,amscd}
\usepackage{textcomp}
\newtheorem{defn}{Definition}
\newtheorem{thm}{Theorem}[section]
\newtheorem{prop}[thm]{Proposition}
\newtheorem{lemme}[thm]{Lemma}
\newtheorem{remark}[defn]{Remark}

\newtheorem*{main}{Main Theorem}
\newtheorem*{str}{Structure Theorem}
\newtheorem*{1.18}{Theorem 1.18 in [DPS]}
\newtheorem*{voin}{Proposition 3.3 in [Voi05]}
\newtheorem*{voin'}{Version B of Proposition 3.3 in [Voi05]}
\newtheorem*{Lef}{Hard Lefschetz theorem}
\newtheorem*{Remark}{Remark}

\newtheorem{cor}[thm]{Corollary}
\newtheorem*{Main theorem}{Main theorem}

\DeclareMathOperator{\Hom}{Hom}

\DeclareMathOperator{\Ch}{Ch}

\DeclareMathOperator{\Alb}{Alb}
\DeclareMathOperator{\Aut}{Aut}
\DeclareMathOperator{\Todd}{Todd}
\DeclareMathOperator{\nd}{nd}
\DeclareMathOperator{\pr}{pr}

\DeclareMathOperator{\inv}{inv}
\DeclareMathOperator{\id}{id}

\title{DEFORMATION OF KäHLER MANIFOLDS}

\author{Junyan CAO}
\email{junyan.cao@ujf-grenoble.fr}
\address{Université de Grenoble I, Institut Fourier, 38402 Saint-Martin d'Hères, France}

\begin{document}

\begin{abstract}
It has been shown by Claire Voisin in 2003 that one cannot always
deform a compact Kähler manifold into a projective algebraic manifold,
thereby answering negatively a question raised by Kodaira. In this article,
we prove that under an additional semipositivity or seminegativity
condition on the canonical bundle, the answer becomes positive, namely
such a compact Kähler manifold can be approximated by deformations of
projective manifolds.
\end{abstract}

\maketitle

\tableofcontents

\section{Introduction}

It is well known that the curvature of the canonical bundle controls the structure of projective varieties. 
C.Voisin has given a counterexample to the Kodaira conjecture 
which states that one cannot always deform a compact Kähler manifold to a projective manifold. 
In her counterexample one can see that the canonical bundle is neither nef nor
anti-nef.
Therefore it is interesting to ask whether for a Kähler manifold with a nef or
anti-nef canonical bundle,
one can deform it to a projective variety. 
In this article, we discuss the deformation properties of Kähler manifolds in
the following three cases:

$(1)$ Compact Kähler manifolds with hermitian semipositive anticanonical bundles.

$(2)$ Compact Kähler manifolds with real analytic metrics and nonpositive bisectional curvatures.

$(3)$ Compact Kähler manifolds with nef tangent bundles.
\vspace{10 pt}

We first recall some definitions about numerical effective (nef) bundles (cf.
\cite{DPS94} for details).
\begin{defn}
A vector bundle $E$ is said to be numerically effective (nef) if the canonical
bundle
$\mathcal{O}_{E}(1)$ is nef on $\mathbb{P}(E)$,
the projective bundle of hyperplanes in the fibres of $E$. 
For a nef line bundle $L$ on a compact Kähler manifold, 
the numerical dimension $\nd (L)$ is defined to be the largest number $v$, such
that $c_{1}(L)^{v}\neq 0$.
A holomorphic vector bundle $E$ over $X$ is said to be numerically flat if both
$E$ and $E^{*}$ are nef 
( or equivalently if $E$ and $(\det E)^{-1}$ are nef).
\end{defn}

\begin{defn}
Let $X$ be a compact Kähler manifold. 
We say that $X$ can be approximated by projective varieties, 
if there exists a deformation of $X$: $\mathcal{X}\rightarrow \Delta$ 
such that the central fiber $X_{0}$ is $X$,
and there exists a sequence $t_{i}\rightarrow 0$ in $\Delta$ such that all the
fibers $X_{t_{i}}$ are projective. 
\end{defn}

The main result of this article is
\begin{main}
If $X$ is a compact Kähler manifold in one of the above three cases, 
then $X$ can be approximated by projective varieties.
\end{main}

The proof for these three types of manifolds relies on their respective
structure theorems. 
We first sketch the strategy of the proof when $X$ is a compact Kähler manifold
with hermitian semipositive anticanonical bundle.
We first recall that
a compact Kähler manifold $X$ is said to be deformation unobstructed, 
if there exists a smooth deformation of $X$,
$\pi:\mathcal{X}\rightarrow \Delta$, such that the Kodaira-Spencer map 
$T_{\Delta}\rightarrow H^{1}(X, T_{X})$ is surjective.
For this type of manifolds, we have the following proposition:

\begin{voin}
Assume that a deformation unobstructed
compact Kähler manifold $X$ has a Kähler
class $\omega$
satisfying the following condition: the interior product
$$\omega\wedge: H^{1}(X, T_{X})\rightarrow H^{2}(X, \mathcal{O}_{X})$$
is surjective. Then $X$ can be approximated by projective varieties.
\end{voin}

In \cite{DPS96}, it is proved that after a finite cover, 
a compact Kähler manifold
with hermitian semipositive anticanonical bundle
has a smooth fibration to a compact Kähler manifold
with trivial canonical bundle and the fibers $Y_{t}$ satisfy the vanishing property:
$$H^{q}(Y_{t},\mathcal{O}_{Y_t})=0\qquad \text{for } q\geq 1.$$
Therefore the Dolbeault cohomology of $X$ is easy to calculate.
One can thus construct explicitly a deformation of $X$ satisfying the surjectivity in
\cite[Proposition 3.3]{Voi1}.
Therefore this type of manifolds can be approximated by projective varieties.
\vspace{10 pt}

When $X$ is a compact Kähler manifold with nef tangent bundle, the proof is more difficult. 
It is based on the structure theorem \cite[Theorem 3.14]{DPS94} which can be stated as follows.

\begin{thm}
Let $X$ be a compact Kähler manifold with nef tangent bundle $T_{X}$. 
Let $\widetilde{X}$ be a finite étale cover of $X$ of maximum irregularity
$q=h^{1}(\widetilde{X}, \mathcal{O}_{\widetilde{X}})$.
Then the Albanese map $\pi: \widetilde{X}\rightarrow T$ is a smooth fibration
over a $q$-dimensional torus, and $-K_{\widetilde{X}}$ is relatively ample.
\end{thm}

\begin{Remark}
We will prove that after passing to 
some finite étale Galois cover $\widetilde{X}\rightarrow X$ with group $G$,
there exists a commutative diagram
$$\begin{CD}
  \widetilde{X} @>>> X\\
 @VV\widetilde{\pi} V @VV\pi V\\
  T @>>> T/ G
  \end{CD}
$$
and $T/ G$ is smooth.
\end{Remark}
 
In \cite[Theorem 3.20]{DPS94}, when $X$ is a projective variety with nef tangent bundle, 
it is proved that $\pi_{*}(-mK_{X})$ is numerically flat for all $m\geq 1$. 
One of the main ingredient of this article is to prove that this is also true when $X$
is a compact Kähler manifold. 

\begin{thm}
Let $X$ be a compact Kähler manifold of dimension $n$ with nef tangent bundle
such that the Albanese map $\pi: X\rightarrow T$ is a smooth fibration onto a torus $T$ of dimension $r$,
and $-K_{X}$ is relatively ample. 
Then $\nd(-K_{X})=n-r$, and $\pi_{*}(-mK_{X})$ is numerically flat for all $m\geq 1$.
\end{thm}

\hspace{-12 pt}We combine this with \cite[Corollary 3.10]{Sim} which states that  
any numerically flat bundle over a compact  Kähler manifold is in fact a local system
\footnote{If the base manifold is a torus,
an explicite construction of the local system would be found in the author's forthcoming Phd thesis.}
:
\begin{thm}
Let $E$ be a numerically flat holomorphic vector bundle on a Galois quotient of a torus $T$, 
then the transformation matrices can be choosen to be constant matrices.
\end{thm}

\hspace{-12 pt}Using Theorem 1.2 and 1.3, 
we will see that one can approximate 
Kähler manifolds with nef tangent bundles by projective varieties.
\vspace{10 pt}

We now sketch the proof of Theorem 1.2. Thanks to a formula in \cite{Be09}, $\pi_{*}(-m K_{X})$ is nef.
Then using the argument in \cite{DPS94}, the only difficult part is to prove $\nd(-K_{X})=n-r$.
If $X$ is projective, the equality $\nd(-K_{X})=n-r$ comes directly from the Kawamata-Viehweg vanishing theorem. 
Since $X$ is just a compact Kähler manifold in our case, 
the proof is more difficult.
We get it by contradiction.
Let $\pi: X\rightarrow T$ be the fibration in Theorem 1.2.
If $\nd(-K_{X})\geq n-r+1$, there are two cases:

\hspace{-12pt}$(i)$ The (1,1)-class $\pi_{*}((-K_{X})^{n-r+1})$ is trivial on
$T$. 

\hspace{-12pt}$(ii)$ The (1,1)-class $\pi_{*}((-K_{X})^{n-r+1})$ is effective (non trivial) on $T$.

In the case $(i)$, thanks to \cite[Corollary 2.6]{DPS94},
we can prove that $\pi_{*}(-mK_{X})$ is numerically flat. 
By Theorem 1.3,
we can thus deform $X$ to a projective manifold 
by preserving $\nd (-K_{X})$. 
Using the Kawamata-Viehweg vanishing theorem in the projective case, 
we can therefore prove that $\nd(-K_{X})=n-r$. 
Thus we get a contradiction.

In the case $(ii)$, the argument is more complicated. 
By solving a Monge-Ampère equation, we can prove that  
$-K_{X}-c\pi^{*}(\pi_{*}((-K_{X})^{n-r+1}))$
is pseudo-effective for some $c>0$. 
Therefore we can construct a singular metric $h$ on $-K_{X}$ 
with a good control on its eigenvalues and with
$\mathcal{I}(h)=\mathcal{O}_{X}$, 
where $\mathcal{I}(h)$ is the multiplier ideal sheaf associated to the singular metric $h$ 
(cf. \cite{Dem12} for the definition of multiplier ideal sheaf).
Thanks to the construction of the metric $h$, we can prove that 
$$H^{r}(X, K_{X}\otimes (-K_{X})\otimes\mathcal{I}(h))=0,$$
where $r=\dim T$.
Therefore $H^{r}(X, \mathcal{O}_{X})=0$, 
which implies that $H^{r}(T, \mathcal{O}_{T})=0$
by the observation that $-K_{X}$ is relatively ample.
Since the torus $T$ is of dimension $r$,
we get a contradiction.
\vspace{5 pt}

The organization of the article is as follows. 
Let $\pi: \widetilde{X}\rightarrow T$ be the smooth fibration of Theorem 1.2. 
In Section 2, we gather some useful propositions.
In particular, we prove a nefness result by using 
\cite[Formula (4.8)]{Be09}. 
In Section 3, we prove our main theorem when $X$ is in the case $(1)$ or $(2)$. 
As an interesting application, the dual cone conjecture in \cite[Conjecture 2.3]{BDPP04} is proved
for the case $(1)$. 
In the following sections, we concentrate on the proof of our main theorem when
$X$ is a compact Kähler manifold with nef tangent bundle.
In Section 4, we prove a deformation lemma which allows us to 
deform a Kähler manifold to a projective variety under certain conditions
and discuss how one can deform $X$ to a projective variety by keeping the
numerical dimension.
In Section 5, we prove a very special Kawamata-Viehweg vanishing theorem which
will play a central role 
in the proof of Theorem 1.2.
Using the results in Section 4 and 5, we finally complete the proof of our main
theorem in Section 6.   
\vspace{5 pt}

\hspace{-12pt}\textbf{Acknowledgements:} 
I would like to thank my supervisor J-P.Demailly for helpful discussions and his
kindness in sharing his ideas.
I would also like to thank C.Voisin 
who explained to me that \cite[Proposition 3.3]{Voi1} could be used to prove
certain approximation problems during a summer school in Norway,
and C.Simpson who told me that the results in \cite{Sim} could largely simplify
the original proof of Theorem 1.3.

\section{Preliminaries}

We first prove some preparatory propositions which are useful in the proof of our main theorem.

\begin{prop}\label{berndslemma}
Let $X$ be a compact Kähler manifold possessing a smooth submersion $\pi: X\rightarrow T$  
to a compact Kähler manifold $T$. 
If $-K_{X}$ is nef on $X$ and is relatively ample for $\pi$, 
then the direct image $E=\pi_{*}(K_{X/T}-(m+1)\cdot K_{X})$ is a nef vector bundle 
for all $m\in\mathbb{N}$.
\end{prop}

\begin{proof}

Let us first show that the direct image $E$ is locally free. 
Let $X_{t}$ be the fiber of $\pi$ over $t\in T$. 
Thanks to the Kodaira vanishing theorem, we have
$$H^{q}(X_{t}, -mK_{X_{t}})=0\qquad \text{for}\hspace{5 pt} q\geq 1.$$ 
By the Riemann-Roch theorem, 
$$\sum_{q} (-1)^{q}h^{q}(X_{t}, -mK_{X_{t}})$$
is a constant independent of $t$. 
Therefore $h^{0}(X_{t}, -mK_{X_{t}})$ is also a constant and 
by a standard result of H.Grauert,
the direct image $E=\pi_{*}(K_{X/T}-(m+1)\cdot K_{X})$ is locally free. 

Since $-(m+1)K_{X}$ is also nef, for any $ \epsilon> 0$ fixed, 
there exists a smooth metric $\varphi$ on $-(m+1)\cdot K_{X}$ such that 
$$i\Theta_{\varphi}(-(m+1)\cdot K_{X})\geq -\epsilon\omega_{T}.$$
Since $E$ is known to be locally free, 
we can use \cite[Formula (4.8)]{Be09}. 
In particular, the metric $\varphi$ on $-(m+1)\cdot K_{X}$ induces a metric on $E$ and
we write its curvature as
$$\Theta^{E}=\sum_{j,k} \Theta_{jk,\varphi}^{E} d t_{j}\wedge d
\overline{t}_{k}$$
where $\{t_{i}\}$ are the coordinates of $T$.
Using the terminology in \cite{Be09}, we assume that
$\{u_{i}\}$ is a base of local holomorphic
sections of $E$ such that $D^{1,0}u_{i}=0$ at a given point. 
We now calculate the curvature at this point. 
Let
$$T_{u}=\sum_{j,k} (u_{j}, u_{k}) \widehat{d t_{j}\wedge d \overline{t}_{k}}.$$
Then 
$$i\partial\overline{\partial}T_{u}=-\sum_{j,k} (\Theta_{jk,\varphi}^{E} u_{j}, u_{k}) d V_{t}.$$
By \cite[Formula (4.8)]{Be09}, we obtain\footnote{The
$i\partial\overline{\partial}\varphi$ below is just
$i\Theta_{\varphi}(-mK_{X})$.}
$$-i\partial\overline{\partial}T_{u}\geq c \pi_{*}(\widehat{u}\wedge
\overline{\widehat{u}}\wedge (i\partial\overline{\partial}\varphi)\cdot e^{-\varphi})$$
where the constant $c$ is independent of $\varphi$.
Since $i\partial\overline{\partial}\varphi\geq -\epsilon\omega_{T}$ by the
choice of $\varphi$, 
we have
$$-i\partial\overline{\partial}T_{u}\geq -c\epsilon\pi_{*}(\widehat{u}\wedge
\overline{\widehat{u}}\wedge\omega_{T}\cdot e^{-\varphi})$$
$$=-c\epsilon (\int_{X_{t}} \sum_{j} (u_{j}, u_{j})e^{-\varphi}) dV_{t}$$
$$=-c\epsilon \|u\|^{2} dV_{t}.$$
In other words, we have
$$\sum_{j,k} (\Theta_{jk,\varphi}^{E} u_{j}, u_{k}) \geq -c\epsilon \|u\|^{2}.$$
Proposition \ref{berndslemma} is proved.

\end{proof}

\begin{prop}
Let $T=\mathbb{C}^{n}/ \Gamma$ be a complex torus of dimension $n$, 
and $\alpha\in H^{1,1}(T,\mathbb{Z})$ an effective non trivial element. 
Then $T$ possess a submersion 
$$\pi: T\rightarrow S$$
to an abelian variety $S$.
Moreover $\alpha=\pi^{*}c_{1}(A)$ for some ample line bundle $A$ on $S$.
\end{prop}

\begin{proof}
Since $T$ is a torus, we can suppose that $\alpha$ is a constant semipositive $(1,1)$-form.
As $\alpha$ is an integral class, 
it defines a bilinear form
$$G_{\mathbb{Q}}: (\Gamma\otimes\mathbb{Q})\times(\Gamma\otimes\mathbb{Q})\rightarrow \mathbb{Q}.$$
We denote its extension to $\Gamma\otimes\mathbb{R}$ by $G_{\mathbb{R}}$.
Let $V$ be the maxium subspace of $\Gamma\otimes\mathbb{Q}$, 
on which $G_{\mathbb{Q}}$ is zero. 
Therefore $V_{\mathbb{R}}=V\otimes\mathbb{R}$ is also the kernel of
$G_{\mathbb{R}}$, and
$(\Gamma\cap V_{\mathbb{R}})\otimes \mathbb{R}=V_{\mathbb{R}}$.
Moreover, since $\alpha$ is an $(1,1)$-form, 
$V_{\mathbb{R}}$ is a complex subspace of $\mathbb{C}^{n}$.
Hence $V_{\mathbb{R}}/ (\Gamma \cap V_{\mathbb{R}})$ is a complex torus.
We denote it $T_{1}$.
Observing that $T/T_{1}$ is also a complex torus, 
we have thus a natural holomorphic submersion $T\rightarrow T/T_{1}$. 
We denote the complex torus $T/T_{1}$ by $S$.
Since $V_{\mathbb{R}}$ is the kernel of $G_{\mathbb{R}}$,
$\alpha$ is well defined on $S$ and is moreover strictly positive on it. 
The proposition is proved.
\end{proof} 

\begin{prop}
Let $E$ be a numerically flat bundle on a compact Kähler manifold. 
Then $E$ is a local system.
\end{prop}

\begin{proof}
Thanks to \cite[Thoerem 1.18]{DPS94}, all numerically flat vector bundles are
successive extensions of
hermitian flat bundles.
By \cite[Corollary 3.10]{Sim}, all such types of bundles are local systems.
The proposition is proved. 
\end{proof}

\begin{remark}
This simple proof is due to C.Simpson. 
When $X$ is just a finite étale quotient of a torus, 
one can give a more elementary proof.
Since that proof is a bit long and technical, 
we omit the proof here and refer instead to our forthcoming PhD thesis. 
\end{remark}

We need a partial vanishing theorem with multiplier ideal sheaf 
(cf.\cite{Dem12} for the definition of multiplier ideal sheaves and analytic
singularities).

\begin{prop}\label{basicvanishing}
Let $L$ be a line bundle on a compact Kähler manifold $(X,\omega)$ of dimension $n$ and 
let $\varphi$ be a metric on $L$ with analytic singularities. 
Let 
$ \lambda_{1}(z)\leq \lambda_{2}(z)\leq\cdots\leq\lambda_{n}(z)$
be the eigenvalues of $\frac{i}{2\pi}\Theta_{\varphi}(L)$ with respect to $\omega$.
If
\begin{equation}\label{strictpositive}
\sum_{i=1}^{p} \lambda_{i}(z)\geq c
\end{equation}
for some constant $c> 0$ independent of $z\in X$,
then
$$H^{q}(X, K_{X}\otimes L\otimes \mathcal{I}(\varphi))=0 
\qquad \text{for }q\geq p.$$
\end{prop}

\begin{proof}

Since $\varphi$ has analytic singularities, 
there exists an analytic subvariety $Y$
such that
$\varphi$ is smooth on $X\setminus Y$.
Moreover it is known that there exists a quasi-psh function $\psi$ on $X$, 
smooth on $X\setminus Y$ 
such that 
$\mathcal{I}(\varphi)=\mathcal{I}(\varphi+\psi)$
and
$\widetilde{\omega}=c_{1}\omega+ i\partial\overline{\partial}\psi$
is a complete metric on $X\setminus Y$ for some fixed constant $c_{1}$ 
with $0<  c_{1}\ll c$ (cf. \cite[Section 5, 6 , Chapter VIII]{Dem}).
To prove the proposition, it is therefore equivalent to prove that
\begin{equation}\label{equvvanishing}
H^{q}(X, K_{X}\otimes L\otimes \mathcal{I}(\varphi+\psi))=0 
\qquad \text{for } q\geq p.
\end{equation}

We consider the new metric $\phi=\varphi+\psi$ on $L$ 
(i.e., the new metric is $\|\cdot\|_{\varphi}\cdot e^{-\psi}$).
Then 
\begin{equation}\label{curvaturecal}
\frac{i}{2\pi}\Theta_{\phi}(L)=\frac{i}{2\pi}\Theta_{\varphi}(L)+dd^{c}\psi
=(\frac{i}{2\pi}\Theta_{\varphi}(L)-c_{1}\omega)+\widetilde{\omega}.
\end{equation}
Since $\varphi$ is a quasi-psh function, there exists a constant $M$
such that
\begin{equation}\label{add2}
\frac{i}{2\pi}\Theta_{\varphi}(L)-c_{1}\omega \geq -M\omega .
\end{equation}
Combining \eqref{add2} with \eqref{curvaturecal}, we obtain
\begin{equation}\label{addd2}
\frac{i}{2\pi}\Theta_{\phi}(L)\geq -M\omega + \widetilde{\omega}.
\end{equation}
Let $\omega_{\tau}=\omega+\tau\widetilde{\omega}$. 
We claim that the sum of $p$-smallest eigenvalues of $
\frac{i}{2\pi}\Theta_{\phi}(L)$ 
with respect to $\omega_{\tau}$ 
is larger than $\frac{c}{2}$, 
for any $\tau\leq\frac{c_1}{1000(M+c)\cdot n\cdot (1+c_1)}$.

Proof of the claim:
By the minimax principle, 
it is sufficient to prove that for any $p$-dimensional subspace $V$ of $(T_{X})_{x}$, 
we have
\begin{equation}\label{sumestimate}
\sum_{i=1}^{p}\langle \frac{i}{2\pi}\Theta_{\phi}(L) e_{i}, e_{i}\rangle\geq \frac{c}{2}
\end{equation}
where $\{e_{i}\}_{i=1}^{n}$ is an orthonormal basis of $V$ with respect to $\omega_{\tau}$.

We first consider the case when $V$ contains an element $e$ such that
\begin{equation}\label{add1}
\widetilde{\omega}(e, e)\geq \frac{c_{1}}{\tau} \qquad\text{and}\qquad |e|_{\omega}=1 .
\end{equation}
By the choice of $\tau$, we have
\begin{equation}\label{addd1}
\widetilde{\omega}(e, e)\geq 1000 n\cdot (M+c).
\end{equation}
Thanks to \eqref{addd2} and \eqref{addd1}, 
we have
$$\langle \frac{i}{2\pi}\Theta_{\phi}(L) e, e\rangle
\geq -M + \widetilde{\omega}(e , e)
\geq \frac{\widetilde{\omega}(e , e)}{2}.$$
Observing moreover that the construction of $\omega_{\tau}$ implies
$$\langle e,e \rangle_{\omega_{\tau}}\leq 1+\tau\cdot\widetilde{\omega}(e, e) ,$$
then
\begin{equation}\label{curvchange}
\frac{\langle \frac{i}{2\pi}\Theta_{\phi}(L) e, e\rangle}{\langle e,e \rangle_{\omega_{\tau}}}\geq 
\frac{\widetilde{\omega}(e, e)}{2+2\tau\widetilde{\omega}(e, e)}\geq 
\frac{1}{2}\min \{\frac{\widetilde{\omega}(e, e)}{2}, \frac{1}{2\tau} \}\geq n(M+c).
\end{equation}
Noting that \eqref{addd2} implies that 
\begin{equation}\label{add3}
\langle\frac{i}{2\pi}\Theta_{\phi}(L) e', e'\rangle\geq -M \omega (e', e')\geq -M\omega_{\tau}(e', e')
\end{equation}
for any $e'\in V$, 
\eqref{curvchange} and \eqref{add3} imply thus the inequality \eqref{sumestimate}.

In the case when 
$$\tau\cdot\widetilde{\omega}(e, e)\leq c_{1}\qquad\text{for any } e\in V \text{ with } |e|_{\omega}=1 ,$$
we have
\begin{equation}\label{difference}
|\omega_{\tau}-\omega |_{\omega}\leq c_{1} \qquad \text{on } V,
\end{equation}
i.e., for considering the restriction on $V$, 
the difference between $\omega_{\tau}|_V$ and $\omega|_V$ is controled by $c_{1}\omega$.
On the other hand, using again the minimax principle, 
\eqref{strictpositive} implies that 
\begin{equation}\label{condition1}
\sum_{i=1}^{p}\langle \frac{i}{2\pi}\Theta_{\varphi}(L) \widetilde{e}_{i}, \widetilde{e}_{i}\rangle\geq c
\end{equation}
for any orthonormal basis $\{\widetilde{e}_{i}\}$ of $V$ with respect to $\omega$. 
By \eqref{curvaturecal}, we have
\begin{equation}\label{addlast}
\frac{i}{2\pi}\Theta_{\phi}(L)\geq (\frac{i}{2\pi}\Theta_{\varphi}(L)-c_{1}\omega).
\end{equation}
Combining \eqref{addlast} with \eqref{condition1} and the smallness assumption on $c_{1}$, 
we have
\begin{equation}\label{finalestimate}
\sum_{i=1}^{p}\langle \frac{i}{2\pi}\Theta_{\phi}(L) \widetilde{e}_{i}, \widetilde{e}_{i}\rangle
\geq \frac{3c}{4}.
\end{equation}
Using again that $c_{1}$ is a fixed constant small enough with respect to $c$, 
\eqref{difference} and \eqref{finalestimate} imply the inequality \eqref{sumestimate}. 
The claim is proved.
\vspace{10 pt}

Let $f$ be a $L$-valued closed $(n,q)$-form such that 
$$\int_{X}|f|^{2}e^{-2\varphi-2\psi}\omega^{n}< +\infty .$$
To prove Proposition \ref{basicvanishing}, 
it is equivalent to find a $L$-valued $(n,q-1)$-form $g$ such that
$$f=\overline{\partial}g \qquad\text{and}\qquad \int_{X}|g|^{2}e^{-2\varphi-2\psi}\omega^{n}< +\infty .$$
Thanks to our claim, we can use the standard $L^{2}$ estimate on 
$$(X\setminus Y , \omega_{\tau}, L, e^{-\varphi-\psi}) .$$
It is known that 
\begin{equation}\label{estimationchangingmetric}
\int_{X\setminus
Y}|f|^{2}e^{-2\phi}\omega_{\tau}^{n}\leq 
\int_{X\setminus
Y}|f|^{2}e^{-2\phi}\omega^{n}< +\infty .
\end{equation}
Then we can find a $g_{\tau}$ such that
$f=\overline{\partial}g_{\tau}$ and
$$\int_{X\setminus Y}|g|^{2}e^{-2\phi}\omega_{\tau}^{n}\leq C
\int_{X\setminus Y}|f|^{2}e^{-2\phi}\omega_{\tau}^{n}< +\infty .$$
for a constant $C$ depending only on $c$ (i.e., $C$ is independent of $\tau$).
Letting $g=\lim\limits_{\tau\rightarrow 0} g_{\tau}$, by \eqref{estimationchangingmetric},
we get
$f=\overline{\partial}g$ on $X\setminus Y$ and
$$\int_{X\setminus Y}|g|^{2}e^{-2\phi}\omega^{n}< +\infty .$$ 
\cite[Lemma 11.10]{Dem12} implies that such $g$ can be extended to the whole
space $X$, and
$f=\overline{\partial}g $ on $X$.
Therefore \eqref{equvvanishing} is proved.
\end{proof}

As a corollary of \cite[Theorem 3.14]{DPS94}, we prove that
every compact Kähler manifold with nef tangent bundle
admits a smooth fibration to an étale Galois quotient of a torus.
\begin{lemme}
Let $X$ be a compact Kähler manifold with nef tangent bundle
and let $\widetilde{X}\rightarrow X$ be an étale Galois cover with group $G$
such that $\widetilde{X}$ satisfies Theorem 1.1 (i.e.,
\cite[Theorem 3.14]{DPS94}). Then $G$ induces a free automorphism group on $T=\Alb(\widetilde{X})$ 
and we have the following commutative diagram
$$\begin{CD}
  \widetilde{X} @>>> X\\
 @VV\widetilde{\pi} V @VV\pi V\\
  T @>>> T/ G
  \end{CD}
$$
where $\widetilde{\pi}: \widetilde{X} \rightarrow T$ is the Albanese map in Theorem 1.1, 
and $T/ G$ is an étale Galois quotient of the torus $T$.
\end{lemme}

\begin{proof}
By the universal property of Albanese map, 
for any $ g\in G$, $g$ induces an automorphism on $T$, 
and the action of $g$ on $\widetilde{X}$ maps fibers to fibers.
We need hence only to prove that $G$ acts on $T$ freely.

Suppose by contradiction that $g(t_{0})=t_{0}$ for some $t_{0}\in T$ and $g\in G$.
Let $\langle g\rangle$ be the subgroup generated by $g$.
Since $g$ acts on $\widetilde{X}$ without fixed point,
$g$ induces an automorphism on $\widetilde{X}_{t_{0}}$ without fixed points, 
where $\widetilde{X}_{t_{0}}$ is the fiber of $\widetilde{\pi}$ over $t_{0}$.
By the same reason, any non trivial elements in $\langle g\rangle$ induces 
an automorphism on $\widetilde{X}_{t_{0}}$ without fixed points.
Combining this with the fact that $\widetilde{X}_{t_{0}}$ is a Fano manifold, 
the quotient $\widetilde{X}_{t_{0}}/ \langle g \rangle$ is hence also a Fano manifold.
Thus the Nadel vanishing theorem implies that
\begin{equation}\label{alternatesum}
\chi(\widetilde{X}_{t_{0}},
\mathcal{O}_{\widetilde{X}_{t_{0}}})=\chi(\widetilde{X}_{t_{0}}/ \langle g \rangle,
\mathcal{O}_{\widetilde{X}_{t_{0}}/\langle g \rangle})=1 .
\end{equation}
\eqref{alternatesum} contradicts with the fact that 
the étale cover $\widetilde{X}_{t_{0}}\rightarrow \widetilde{X}_{t_{0}}/ \langle g \rangle$
implies
$$\chi(\widetilde{X}_{t_{0}},
\mathcal{O}_{\widetilde{X}_{t_{0}}})=|\langle g \rangle|\cdot \chi(\widetilde{X}_{t_{0}}/ \langle g \rangle,
\mathcal{O}_{\widetilde{X}_{t_{0}}/\langle g \rangle}) .$$
Then $G$ factorizes to an étale Galois action on $T$, and the lemma is proved.

\end{proof}

\section{Deformation of compact Kähler manifolds with hermitian semipositive anticanonical bundles 
or nonpositive bisectional curvatures}

We first treat a special case, 
i.e., how to appproximate compact manifolds with numerically trivial canonical bundles by projective varieties.
To prove the statement, we need the following two propositions.

\begin{voin}
Assume that a deformation unobstructed compact Kähler manifold $X$ has a Kähler
class $\omega$
satisfying the following condition: the interior product
$$\omega\wedge: H^{1}(X, T_{X})\rightarrow H^{2}(X, \mathcal{O}_{X})$$
is surjective. Then $X$ can be approximated by projective varieties.
\end{voin}

\begin{remark}
The proof of this proposition is based on a density criterion (cf. \cite[Proposition 5.20]{Voi2})
which will also be used in Proposition 3.3 and  Proposition 3.5.
We need moreover a slightly generalized version of \cite[Proposition 3.3]{Voi1}. 
In fact, we can suppose $\omega$ to be a nef class in $X$, 
since the surjectivity is preserved under small perturbation.
Moreover, if $X$ is not necessarily unobstructed, 
we just need a deformation unobstructed subspace $V$ of $H^{1}(X, T_{X})$ such
that
$$\omega\wedge V\rightarrow H^{2}(X, \mathcal{O}_{X})$$
is surjective.
In summary, we have the following variant of the above proposition.
\end{remark}

\begin{voin'}
Let $\mathcal{X}\rightarrow \Delta$ be a deformation of a compact Kähler manifold $X$ 
and $V$ be the image of Kodaira-Spencer map of this deformation.
If there exists a nef class $\omega$ in $H^{1,1}(X)$ such that
$$\omega\wedge V\rightarrow H^{2}(X, \mathcal{O}_{X})$$
is surjective,
then there exists a sequence $t_{i}\rightarrow 0$ in $\Delta$ 
such that all the fibers $X_{t_{i}}$ are projective.
\end{voin'}

In general, it is difficult to check the surjectivity in the above proposition.
By a well-known observation communicated to us by J-P. Demailly, one can prove that
the above morphism is surjective when $-K_{X}$ is hermitian semipositive by
using the following Hard Lefschetz theorem.

\begin{Lef}
(cf. \cite[Corollary 15.2]{Dem12})
Let $(L,h)$ be a semi-positive line bundle on a compact Kähler manifold $(X,\omega)$
of dimension $n$ i.e., $h$ is a smooth metric on $L$ and $i\Theta_{h}(L)\geq 0$.
Then the wedge multiplication operator $\omega^{q}\wedge$ induces
a surjective morphism
$$\omega^{q}\wedge: H^{0}(X, \Omega_{X}^{n-q}\otimes L)\rightarrow H^{q}(X, \Omega_{X}^{n}\otimes L).$$
\end{Lef}

Using the above two propositions, we can reprove the following well-known fact.
\begin{prop}\label{easyappli}
Let $X$ be a compact Kähler manifold with $c_{1}(X)_{\mathbb{R}}=0$. 
Then it can be approximated by projective varieties.
\end{prop}

\begin{proof}
By a theorem due to Beauville, there exists a finite Galois cover $\widetilde{X}\rightarrow X$ 
such that $K_{\widetilde{X}}$ is trivial. 
Then $K_{X}$ is a torsion line bundle. 
Using the Tian-Todorov theorem (cf. the torsion version in \cite{Ran}), $X$ is
unobstructed.
To prove Proposition \ref{easyappli}, by \cite[Proposition 3.3]{Voi1}, 
we just need to check that 
\begin{equation}\label{surjecitvedensity}
\omega\wedge: H^{1}(X, T_{X})\rightarrow H^{2}(X, \mathcal{O}_{X})
\end{equation}
is surjective for some Kähler class $\omega$.

In fact, since $c_{1}(K_{X})_{\mathbb{R}}=0$, 
there exists a smooth metric $h$ on $-K_{X}$ such that $i\Theta_{h}(-K_{X})=0$.
Thus $( -K_{X}, h )$ is semipositive. 
Then the Hard Lefschetz theorem above told us that for any Kähler metric $\omega$, the morphism
\begin{equation}\label{firstsurject}
\omega^{2}\wedge: H^{0}(X, \Omega_X ^{n-2}\otimes (-K_{X}))\rightarrow H^{2}(X, K_{X}\otimes (-K_{X}))
\end{equation}
is surjective.
Observing moreover that the image of \eqref{firstsurject}
is contained in the image of
$$\omega\wedge H^{1}(X, \Omega_X ^{n-1}\otimes (-K_{X}))=\omega\wedge H^{1}(X, T_{X}),$$
i.e., the image of \eqref{surjecitvedensity}.
Then
\eqref{surjecitvedensity} is surjective.
Using \cite[Proposition 3.3]{Voi1}, the proposition is proved. 
\end{proof}

We now begin to prove the main proposition in this section, i.e.,
one can approximate compact Kähler manifolds with hermitian
semipositive anticanonical bundles by projective varieties.
The main tool is the following structure theorem in \cite{DPS96}:
\begin{str}
Let $X$ be a compact Kähler manifold with $-K_{X}$ hermitian semipositive. Then

$(i)$ The universal cover $\widetilde{X}$ admits a holomorphic and isometric splitting
$$\widetilde{X}=\mathbb{C}^{q}\times Y_{1}\times Y_{2}$$
with $Y_{1}$ being the product of either Calabi-Yau manifolds or symplectic manifolds,
and $Y_{2}$ being projective. Moreover $H^{0}(Y_{2}, \Omega_{Y_{2}}^{\otimes q})=0$ for $q\geq 1$.

$(ii)$ There is a normal subgroup $\Gamma_{1}\subset \pi_{1}(X)$ of finite index 
such that $\widehat{X}=\widetilde{X}/ \Gamma_{1}$ has a smooth fibration to a Ricci-flat compact manifold:
$F=(\mathbb{C}^{q}\times Y_{1}) / \Gamma_{1}$
with fibers $Y_{2}$.
\end{str}

\begin{remark}\label{vanishingoftensor}
Since $\Omega_{Y_{2}}^{q}\subset \Omega_{Y_{2}}^{\otimes q}$,
the above structure theorem implies that
$$H^{0}(Y_{2}, \Omega_{Y_{2}}^{q})=0 .$$
Therefore $H^{q}(Y_{2}, \mathcal{O}_{Y_2})=0$ by duality.
\end{remark}

\begin{remark}\label{fibrationprojecitivity}
The Ricci semipositive metric on $X$ induces a $\pi_1 (X)$-invariant metric $\omega_{Y_2}$ on $Y_2$.
Thanks to Remark \ref{vanishingoftensor}, we can suppose that $\omega_{Y_2}\in H^{1,1}(Y_2, \mathbb{Q})$.
Therefore $\omega_{Y_2}$ induces a rational coefficience, 
closed semipositive $(1,1)$-form on $\widehat{X}$, which is strictly positive on the fibers of the fibration in $(ii)$
of the above Structure Theorem.
\end{remark}

We need also the following lemma.
\begin{lemme}
Let $X$ be a compact Kähler manifold with $K_{X}=\mathcal{O}_{X}$, 
and $G$ a finite subgroup of the biholomorphic group $\Aut (X)$.
Then there exists a local deformation of $X$: $\mathcal{X}\rightarrow \Delta$
such that
the image of the Kodaira-Spencer map of this deformation is equal to $H^{1}(X, T_{X})^{G-\inv}$ and
$\mathcal{X}$ admits a holomorphic $G$-action fiberwise, 
where $H^{1}(X, T_{X})^{G-\inv}$ is the $G$-invariant subspace of $H^{1}(X,
T_{X})$.
\end{lemme}

\begin{proof} 
By the Kuranishi deformation theory, it is sufficient to construct a 
vector valued $(0,1)$-form
$$\varphi(t)=\sum_{k_{i}\geq 0} \varphi_{k_{1}\cdots k_{m}}t_{1}^{k_{1}}\cdots t_{m}^{k_{m}}$$
such that
\begin{equation}\label{deformationcondition}
 \varphi(0)=0 \qquad\text{and}\qquad
\overline{\partial}\varphi(t)=\frac{1}{2}[\varphi(t), \varphi(t)] ,
\end{equation}
where $\varphi_{k_{1}\cdots k_{m}}$ are $G$-invariant vector
valued $(0,1)$-forms,
$\{\varphi_{k_{1}\cdots k_{m}}\}_{\sum k_{i}=1}$ gives a basis of $H^{1}(X, T_{X})^{G-\inv}$
and $t_{1},...,t_{m}$ are parameters of $\Delta$.
By \cite{MK}, solving \eqref{deformationcondition} is equivalent to find $G$-invariant vector valued $(0,1)$-forms
$\varphi_{\mu}$ such that
\begin{equation}\label{deformationconditionequation}
\overline{\partial}\varphi_{\mu}=\frac{1}{2}\sum_{|\lambda|< |\mu|} [\varphi_{\lambda},
\varphi_{\mu-\lambda}]
\end{equation}
for any $\mu$.

Suppose that we have already found $\varphi_{\mu}$ for $|\mu|\leq
N$ such that \eqref{deformationconditionequation} is satisfied for all $|\mu|\leq N$. 
If $|\mu |=N+1$, thanks to \cite{Tian}, there exists a vector valued
$(0,1)$-form $s_{\mu}$
satisfying
$$\overline{\partial}s_{\mu}=\frac{1}{2}\sum_{|\lambda|\leq N} [\varphi_{\lambda}, \varphi_{\mu-\lambda}].$$
Recall that if $Y_{1}$, $Y_{2}$ are two $G$-invariant vector valued $(0,1)$-forms, then
$[Y_{1}, Y_{2}]$ is also a $G$-invariant vector valued $(0,2)$-form
\footnote{Let $\alpha\in G$, $f\in C^{\infty}(X)$ and $x\in X$. Using the
$G$-invariance of $Y_{1}$ and $Y_{2}$, we have
$\alpha^{*}(Y_{1}Y_{2})(f)(x)=Y_{1}Y_{2} (f\circ \alpha)(\alpha^{-1}(x))
=Y_{1} (Y_{2}(f)\circ \alpha)(\alpha^{-1}(x))=Y_{1}(Y_{2}(f))(x)$. 
Thus $[Y_{1}, Y_{2}]$ is also $G$-invariant.}.
Therefore $\overline{\partial}s_{\mu}$ is a $G$-invariant vector
valued $(0,2)$-form.
The finiteness of $G$ and the above $G$-invariance of
$\overline{\partial}s_{\mu}$ 
imply hence that 
$\frac{1}{|G|}\sum\limits_{g\in G}g^{*}s_{\mu}$ is a $G$-invariant vector valued
$(0,1)$-form satisfying \eqref{deformationconditionequation}. 
The lemma is proved.
\end{proof}

The following proposition tells us that for a compact Kähler manifold with 
numerically trivial canonical bundle, 
if it admits ``more automorphisms'', then it is ``more algebraic''. More precisely, we have
\begin{prop}
Let $\pi: \mathcal{X} \rightarrow \Delta$ be the deformation constructed in Lemma 3.2. 
Then there exists a sequence $t_{i}\rightarrow 0\in \Delta$ such that
$X_{t_{i}}$ are projective varieties.
\end{prop}

\begin{proof}
We first prove that $H^{2}(X, \mathbb{Q})^{G-\inv}$ admits a sub-Hodge
structure of $H^{2}(X,\mathbb{Q})$.
In fact, we have the equality 
\begin{equation}\label{hodgeequation}
H^{2}(X, \mathbb{Q})^{G-\inv}\otimes \mathbb{R}=H^{2}(X, \mathbb{R})^{G-\inv}
\end{equation}
by observing that the elements in $G$ act continuous on $H^{2}(X, \mathbb{R})$. 
Combining \eqref{hodgeequation} with the obvious Hodge decomposition
$$H^{2}(X, \mathbb{C})^{G-\inv}=\oplus_{p+q=2}H^{p,q}(X, \mathbb{C})^{G-\inv}, $$
$H^{2}(X, \mathbb{Q})^{G-\inv}$ is thus a sub-Hodge structure of $H^{2}(X,\mathbb{Q})$.
Then $\pi$ induces a VHS of $H^{2}(X, \mathbb{Q})^{G-\inv}$.

Let $\omega_{X}$ be a $G$-invariant Kähler metric on $X$. 
\eqref{surjecitvedensity} of Proposition 3.1 implies that
$$\omega_{X}\wedge H^{1}(X, T_{X})\rightarrow H^{2}(X, \mathcal{O}_X)$$
is surjective. 
Thanks to the $G$-invariance of $\omega_{X}$, 
$$\omega_{X}\wedge H^{1}(X, T_{X})^{G-\inv}\rightarrow H^{2}(X, \mathcal{O}_X )^{G-\inv}$$
is also surjective.
Using the density criterion \cite[Proposition 5.20]{Voi2} 
and the same argument of \cite[Proposition 3.3]{Voi1},
the proposition is proved.
\end{proof}

We now prove the main result in this section.
\begin{thm}
Let $X$ be a compact Kähler manifold with $-K_{X}$ hermitian semipositive.
Then it can be approximated by projective varieties.
\end{thm}

\begin{proof}
We prove it in three steps.

\textbf{Step 1}: We use the terminology of the Structure Theorem in this section.
Let $G=\pi_{1}(X) / \Gamma_{1}$ and $\widehat{X}=\widetilde{X}/ \Gamma_{1}$. 
Then $G$ acts on $\widehat{X}$ and $X=\widehat{X}/ G$. 
Thanks to $(ii)$ of the Structure Theorem in this section,  
we have a smooth fibration
\begin{equation}\label{corofstructthm}
\pi: \widehat{X}\rightarrow F 
\end{equation}
with the fibers $Y_{2}$.
We prove in this step that
\begin{equation}\label{firstproveequa}
H^{q}(\widehat{X},\mathcal{O}_{\widehat{X}})= \pi^* ( H^{q}(F,\mathcal{O}_F) ) 
\end{equation}
and 
\begin{equation}\label{secondproveequa}
H^{q}(\widehat{X},\mathcal{O}_{\widehat{X}})^{G-\inv}=\pi^{*}(H^{q}(F,\mathcal{O}_F)^{G-\inv} ).
\end{equation}

Using the smooth fibration \eqref{corofstructthm},
we can calculate $H^{q}(\widehat{X},\mathcal{O}_{\widehat{X}})$ by the Leray
spectral sequence.
Then \eqref{firstproveequa} comes directly from the fact that
$$H^{q}(Y_{2}, \mathcal{O}_{Y_2})=0 \qquad\text{for } q\geq 1 $$ 
(cf. Remark \ref{vanishingoftensor} of the Structure Theorem in this section).
As for \eqref{secondproveequa}, 
we just need to check that the image of the injective map
\begin{equation}\label{checkinjective}
\pi^{*}: H^{q}(F,\mathcal{O}_F)^{G-\inv}\rightarrow
H^{q}(\widehat{X},\mathcal{O}_{\widehat{X}})
\end{equation}
is $H^{q}(\widehat{X},\mathcal{O}_{\widehat{X}})^{G-\inv}.$
Let $\gamma \in G$ and $\alpha$ a smooth differential form on $F$.
Since $\pi_{1}(X)$ acts on $\mathbb{C}^{q}\times Y_{1}$ and $Y_{2}$ separately,
we have the diagram
$$\begin{CD} \widehat{X}@>\gamma>> \widehat{X}\\
 @VV\pi V @VV\pi V\\ 
F@>\gamma>> F\end{CD}$$
Then the equality
$$\gamma^{*}( \pi^{*}\alpha )=\pi^{*}(\gamma^{*} \alpha)$$
implies that the image of \eqref{checkinjective} is contained in $H^{q}(\widehat{X},\mathcal{O}_{\widehat{X}})^{G-\inv}$.
To prove that $H^{q}(\widehat{X},\mathcal{O}_{\widehat{X}})^{G-\inv}$ is in the image of \eqref{checkinjective}, 
we take an element $\beta\in H^{q}(\widehat{X},\mathcal{O}_{\widehat{X}})^{G-\inv}$.
Thanks to the proved equality \eqref{firstproveequa}, 
we can find an element $\mu\in H^{q}(F,\mathcal{O}_F)$ 
such that $\pi^{*}\mu=\beta$ as an element in $H^{q}(\widehat{X},\mathcal{O}_{\widehat{X}})$.
Since 
$$\pi^{*}(\gamma^{*}\mu)=\gamma^{*}(\pi^{*}\mu)=\gamma^{*}(\beta)=\beta=\pi^{*}(\mu)$$
in $H^{q}(\widehat{X},\mathcal{O}_{\widehat{X}})$, 
the injectivity of \eqref{checkinjective}
implies that $\gamma^{*}(\mu)=\mu$ in $H^{q}(F, \mathcal{O}_F)$.
Then $\mu$ is $G$-invariant. Therefore \eqref{checkinjective} gives an isomorphism from
$H^{q}(F,\mathcal{O}_F)^{G-\inv}$
to
$H^{q}(\widehat{X},\mathcal{O}_{\widehat{X}})^{G-\inv}.$
\eqref{secondproveequa} is proved.

\textbf{Step 2}: Let $\omega_{F}^{G-\inv}$ be a $G$-invariant Kähler metric on $F$.
We construct in this step a deformation of $F$: $\mathcal{F}\rightarrow \Delta$ such that
$$\omega_{F}^{G-\inv}\wedge V\rightarrow H^{2}(F,\mathcal{O}_F)^{G-\inv}$$
is surjective, where $V$ is the image of the Kodaira-Spencer map of this deformation.
Moreover, $\mathcal{F}$ should admit a holomorphic $G$-action fiberwise.

Since $c_{1}(F)_{\mathbb{R}}=0$ by construction, Proposition 3.1 implies that
$$\omega_{F}^{G-\inv}\wedge H^{1}(F, T_{F})\rightarrow H^{2}(F, \mathcal{O}_F)$$
is surjective. Then
\begin{equation}\label{surjectionequation}
\omega_{F}^{G-\inv}\wedge H^{1}(F, T_{F})^{G-\inv}\rightarrow H^{2}(F,
\mathcal{O}_F)^{G-\inv} 
\end{equation}
is also surjective.
Thanks to Lemma 3.2, there exists a deformation of $F$ satisfying the requirements of deformation in this step, 
especially it admits a holomorphic $G$-action fiberwise.

\textbf{Step 3}: Final conclusion.

Since $\widehat{X}$ is the quotient of $\Gamma_{1}\curvearrowright \mathbb{C}^{q}\times Y_{1}\times Y_{2}$
and $\Gamma_{1}$ acts on $\mathbb{C}^{q}\times Y_{1}$ and $Y_{2}$ separately,
the deformation of $F= (\mathbb{C}^{q}\times Y_{1}) / \Gamma_{1}$ in Step 2 induces a deformation of $\widehat{X}$:
$$\widehat{\mathcal{X}}\rightarrow \Delta$$
by preserving the complex structure of $Y_{2}$.
By construction, we have a natural fibration
$$\widetilde{\pi}: \widehat{\mathcal{X}}\rightarrow \mathcal{F} .$$
Moreover, since $G$ is holomorphic for the fibers of $\mathcal{F}$ over $\Delta$,
the quotient $\mathcal{X}=\widehat{\mathcal{X}}/ G$ is a smooth deformation of $X$.
In summary, we have the following diagrams:
$$\begin{CD} \widehat{X}@>G>> X=\widehat{X} / G\\ @VV\pi V\\ F\end{CD}  \qquad\text{and} \qquad 
\begin{CD} \widehat{\mathcal{X}}@>G>> \mathcal{X}=\widehat{\mathcal{X}}/G\\ @VV\widetilde{\pi} V\\ \mathcal{F}\end{CD}.
$$
Let $X_t, F_t$ be the fibers of $\mathcal{X}$, $\mathcal{F}$ over $t\in \Delta$.
Thanks to Proposition 3.3, there exists a sequence $t_{i}\rightarrow 0\in \Delta$
such that $F_{t_{i}}$ are projective. 
Combining this with Remark \ref{fibrationprojecitivity} after 
the Structure Theorem in this section,
we obtain that $X_{t_{i}}$ are projective. 
The proposition is proved.
\end{proof}

\begin{remark}\label{remarkvari}
For the further application, we need to study the deformation $\mathcal{X}$ in detail.  
Let $\pr: \widehat{X}\rightarrow X$ be the quotient.
Since $\pi^{*}\omega_{F}^{G-\inv}$ is a $G$-invariant semipositive form on
$\widehat{X}$, we can find a nef class $\alpha$ on $X$ such that $\pr^* (\alpha)= \pi^{*}\omega_{F}^{G-\inv}$.
We denote it also $\pi^{*}\omega_{F}^{G-\inv}$ for simplicity.
Let $V$ be the image of Kodaira-Spencer map of the deformation $\mathcal{X}\rightarrow \Delta$.
We now prove that
\begin{equation}\label{provingsurj}
\alpha\wedge V\rightarrow H^{2}(X,\mathcal{O}_X)
\end{equation}
is surjective.
Thanks to the construction of $\widehat{\mathcal{X}}$ and the surjectivity of \eqref{surjectionequation},
the morphism 
\begin{equation}\label{surX1}
 \pi^{*}\omega_{F}^{G-\inv}\wedge W\rightarrow \pi^{*} (H^{2}(F, \mathcal{O}_F)^{G-\inv})
\end{equation}
is surjective on $\widehat{X}$, 
where $W$ is the image of Kodaira-Spencer map of the deformation $\widehat{\mathcal{X}}\rightarrow \Delta$.
Since
$$\pi^{*} (H^{2}(F, \mathcal{O}_F)^{G-\inv})=H^{q}(\widehat{X},\mathcal{O}_{\widehat{X}})^{G-\inv}$$
which is proved in Step 1,
\eqref{surX1} implies that
$$\pi^{*}\omega_{F}^{G-\inv}\wedge W\rightarrow H^{q}(\widehat{X},\mathcal{O}_{\widehat{X}})^{G-\inv}$$
is surjective.
Hence \eqref{provingsurj} is surjective.
\end{remark}

As an application, we prove \cite[Conjecture 2.3 and 10.1]{BDPP04}
for compact Kähler manifolds with hermitian semipositive anticanonical bundles.

\begin{prop}
If $X$ is a compact Kähler manifold with $-K_{X}$ hermitian semipositive, 
then the Conjecutre 2.3 and 10.1 in \cite{BDPP04} are all true.
\end{prop}

\begin{proof}
By Remark \ref{remarkvari} after Theorem 3.4, 
there exists a local deformation of $X$
$$\pi: \mathcal{X}\rightarrow \Delta,$$
such that 
\begin{equation}\label{surjectivehodge}
 \alpha\wedge V\rightarrow H^{2}(X,\mathcal{O}_X)
\end{equation}
is surjective for some nef class $\alpha\in H^{1,1}(X, \mathbb{R})$,
where $V$ is the image of the Kodaira-Spencer map of $\pi$.

Let $\beta$ be any smooth closed $(1,1)$-form on $X$. 
Thanks to the surjectivity of \eqref{surjectivehodge}, 
$$(\beta+s \alpha) \wedge V\rightarrow H^{2}(X,\mathcal{O}_X)$$
is also surjective for any $s\neq 0$ small enough.
By the proof of \cite[Proposition 5.20]{Voi2}, we can hence find a sequence of
smooth closed $2$-forms $\{\beta_{t}\}$ on $X$, such that
$$\lim_{t\rightarrow 0} \beta_{t}=\beta+s\alpha$$ 
in $C^{\infty}$-topology and $\beta_{t}\in H^{1,1}(X_{t}, \mathbb{Q})$.
By the same argument as in \cite[Theorem 10.12]{BDPP04}, 
the proposition is proved. 
\end{proof}

We now study the case when $X$ has a real analytic metric and nonpositive
bisectional curvatures.
Recall first the structure theorem \cite[Theorem E]{WZ}

\begin{prop}
Let $X$ be a compact Kähler manifold of dimension $n$ 
with real analytic metric and nonpositive bisectional curvature, 
and let $\widetilde{X}$ be its universal cover. Then

$(i)$ There exists a holomorphically isometric decomposition 
$\widetilde{X}=\mathbb{C}^{n-r}\times Y^{r}$,
where $Y^{r}$ is a complete manifold with nonpositve bisectional curvature 
and the Ricci tensor of $Y^{r}$ is strictly negative somewhere.

$(ii)$ (cf. \cite[Claim 2, Theorem E]{WZ})
There exists a finite index sub-normal group $\Gamma'$ of
$\Gamma=\pi_{1}(X)$
such that $Y^{r}/ \Gamma'$ is a compact manifold and $\widetilde{X}/ \Gamma'$
possess the smooth fibrations to
$Y^{r}/ \Gamma'$ and $\mathbb{C}^{n-r}/\Gamma' $.
\end{prop}

\begin{remark}\label{remarkWZ}
By \cite[Claim 2, Theorem E]{WZ}, $\mathbb{C}^{n-r}/\Gamma' $ is a torus.
We should notice that in contrast to the case when $-K_{X}$ is semipositive, 
$Y^{r}$ is not necessary compact in this proposition. 
The universal covers of curves of genus $g\geq 2$ are typical exemples.
The good news here is that $Y^{r}/ \Gamma'$ is a projective variety of general type
thanks to $(i)$.
\end{remark}

\begin{prop}
Let $X$ be a compact Kähler manifold of dimension $n$ 
with real analytic metric and nonpositive bisectional curvature.
Then it can be approximated by projective varieties. 
\end{prop}

\begin{proof}
Keeping the notation in Proposition 3.6, we know that 
$T=\mathbb{C}^{n-r}/\Gamma'$ 
is a torus with a finite group action $G=\Gamma/ \Gamma'$.
Let $\widehat{X}=\widetilde{X}/ \Gamma'$.
By Lemma 3.2, there exists a deformation of $T$ 
$$\pi: \mathcal{T}\rightarrow \Delta$$
such that 
$G$ acts holomorphically fiberwise.
Therefore this deformation induces 
the deformations of $\widehat{X}$ and $X$ 
by preserving the complex structure on $Y^{r}$.
We denote 
\begin{equation}\label{twodeformation}
\widehat{\mathcal{X}} \rightarrow \Delta\qquad 
\text{and}\qquad \mathcal{X}\rightarrow \Delta .
\end{equation}
Thanks to the construction, $X_{t}$ is the $G$-quotient of $\widetilde{X}_{t}/\Gamma'$,
where $X_{t}$ and $\widetilde{X}_{t}/\Gamma'$ are 
the fibers over $t\in \Delta$ of the above deformations.

Let $t_{i}\rightarrow 0$ be the sequence in Proposition 3.3 such that
$T_{t_{i}}$ are projective. 
By Proposition 3.6, we have two fibrations: 
$$\widehat{X}_{t_{i}} \rightarrow T_{t_{i}} 
\qquad \text{and} \qquad \widehat{X}_{t_{i}} \rightarrow Y^{r}/ \Gamma' .$$
Thanks to the projectivity of $T_{t_{i}}$ and Remark \ref{remarkWZ} of Proposition 3.6,
$\widehat{X}_{t_{i}} $ is thus projective.
Therefore $X_{t_{i}}$ is projective and the proposition is proved.
\end{proof}

\section{A deformation proposition}

The following sections are devoted to the deformation problem of compact
Kähler manifolds with nef tangent bundles.
We discuss in this section how to deform varieties that are defined by certain
numerically flat fibrations.
We first prove a preparatory lemma.

\begin{lemme}
Let $X$ be a compact K\"{a}hler manifold and 
let $E$ be a numerically flat vector bundle on $X$ possessing a filtration
\begin{equation}\label{numericalfiltration}
0=E_{0}\subset E_{1}\subset\cdots\subset E_{m}=E 
\end{equation}
such that the quotients $E_{i}/ E_{i-1}$ are irreducible hermitian flat vector bundles. 
Then $E$ is a local system and 
all elements in $H^{0}(X, E)$ are parallel with respect to the natural local system induced by the filtration \eqref{numericalfiltration}.

In particular, if there are two such filtrations, 
the transformation matrices between these two local systems should be locally constant.
\end{lemme}

\begin{proof}
Thanks to \cite[Corollary 3.10]{Sim}, 
the filtration \eqref{numericalfiltration} induces a natural local system on $E$ and 
the natural Gauss-Manin connection on $E$ preserves the filtration \eqref{numericalfiltration}
(i.e., the connection on each successive quotient $E_i/E_{i-1}$ 
induced by the Gauss-Manin connection on $E$ is the natural hermitian flat connection on $E_i/E_{i-1}$).
Using the recurrence process, 
to prove that all elements in $H^{0}(X, E)$ are parallel with respect to the local system,
it is sufficient to prove that
if $E$ is a non trivial extension 
\begin{equation}\label{extensionsequence}
\begin{CD}
0 @>>> E_{m-1} @>i>>  E 
@>>> E_m /E_{m-1} @>>> 0  
\end{CD} ,
\end{equation}
then $H^{0}(X, E)= i (H^{0}(X, E_{m-1}))$.
To prove this, we first note that \eqref{extensionsequence} implies the exact sequence
$$\begin{CD} 0 @>>> H^{0}(X, E_{m-1})
 @>i>> H^{0}(X, E) @>>>H^{0}(X, E_{m}/E_{m-1})
@>\delta>> H^{1}(X, E_{m-1})\end{CD} .$$

{\em Case 1: $E_{m}/E_{m-1}\neq \mathcal{O}_X$.} 
Since $E_{m}/ E_{m-1}$ is an irreducible hermitian flat bundle,
we have 
\begin{equation}\label{irreduciblenull}
H^{0}(X, E_{m}/E_{m-1})=0.
\end{equation}
Using the above exact sequence, we obtain $H^{0}(X, E)= i (H^{0}(X, E_{m-1}))$.

{\em Case 2: $E_{m}/E_{m-1}= \mathcal{O}_X$ .}
Since $h^{0}(X, \mathcal{O}_{X})=1$ and $E$ is a non trivial extension,
we obtain that $\delta$ in the exact sequence is injective.
Therefore $i(H^{0}(X, E_{m-1}))= H^{0}(X, E)$.
By recurrence, all elements in $H^{0}(X, E)$ should be parallel 
with respect to the natural local system induced by \eqref{numericalfiltration}.

For the second part of the lemma, if there is another filtration
$$0=E'_{0}\subset E'_{1}\subset\cdots\subset E'_{n}=E ,$$
then it induces a filtration on $E^{*}$.
Using this filtration on $E^{*}$ and the filtration \eqref{numericalfiltration} on $E$,
we get a natural filtration on $\Hom (E, E)=E^{*}\otimes E$.
Applying the first part of the lemma, 
the natural identity element 
$\id\in H^{0}(X, \Hom (E, E))$ should be parallel with respect to the filtration.
In other words, 
the transformation matrices between these two filtrations should be locally constant.
\end{proof}

\begin{remark}
We should remark that for a general local system on a compact Kähler manifold, the global sections may not be parallel with respect to the 
flat connection.
\end{remark}

\begin{prop}
Let $X$ be a Kähler manifold possessing a submersion 
$\pi: X\rightarrow T$, 
where $T$ is a finite étale quotient of a torus. 
Assume that $-K_{X}$ is nef and relatively ample.
If moreover $E_{m}=\pi_{*}(-mK_{X})$ is numerically flat for any $m > 0$, 
then there is a smooth deformation of the fibration which can be realized as:
$$\begin{CD} \mathcal{X} @>\pi >> \mathcal{T} @> \pi_{1}>> \Delta \end{CD}$$ 
such that 
$\pi_{1}: \mathcal{T} \rightarrow \Delta$
is the local universal deformation of $T$ and the central fiber is $X\rightarrow
T$.

Moreover, let $T_{s}$ be the fiber of $\pi_{1}$ over $s\in \Delta$,
and let $X_{s}$ be the fiber of $\pi\circ\pi_{1}$ over $s\in \Delta$.
Then the anticanonical bundle of $X_{s}$ is also nef and relatively ample with
respect to the fibration
$X_{s}\rightarrow T_{s}$ for any $s\in \Delta$.
\end{prop}

\begin{proof}

Thanks to \cite[Theorem 3.20]{DPS94}, we have the embeddings $X\hookrightarrow
\mathbb{P}(E_{m})$ 
and $V_{m,p}=\pi_{*}(\mathcal{I}_{X}\otimes\mathcal{O}_{\mathbb{P}(E_{m})}(p))\subset S^{p}E_{m}$
for $m, p$ large enough.
More importantly, $V_{m,p}$ and  $S^{p}E_{m}$ are numerically flat vector bundles.
By the proof of Lemma 4.1, $ S^{p}E_{m}$ is in fact a local system on $T$ 
which be represented by locally constant transformation matrices
and its subbundle $V_{m,p}$ can be represented by the upper blocks of the transformation matrices. 

Thanks to \cite[Proposition 2.3]{Ran}, the deformation of $T$ is
unobstructed. 
Let $\pi_{1}: \mathcal{T}\rightarrow \Delta$ be the local universal deformation of $T$. 
Then the transformation matrices of  $S^{p}E_{m}, V_{m,p}$ 
are always holomorphic under the deformation of the complex structure on $T$.
Therefore we get the holomorphic deformations of these vector bundles by
changing the complex strucutre on $T$: 
$$\xymatrix{
\mathcal{V}_{m,p}\ar[rd]\ar@{^{(}->}[r] & S^{p}\mathcal{E}_{m}\ar[d] \\
& \mathcal{T}\ar[d]\\
&\Delta}
\qquad\text{and}\qquad
\xymatrix{
\mathcal{V}_{m,p}\times \mathbb{P}(\mathcal{E}_{m})\ar[d] \\
\mathcal{T}\ar[d]\\
\Delta} .
$$

We first note that, by the discussion after \cite[Proposition 3.19]{DPS94},
a local basis of $V_{m,p}$ 
gives the determinant polynomials of $X$ in $\mathbb{P}(E_{m})$ over $U$.
Now we have two filtrations on $S^{p}E_{m}$, 
one is induced by the inclusion $V_{m,p}\subset S^{p}E_{m}$ 
and the another is induced by a filtration on $E_{m}$.
Thanks to Lemma 4.1, on any small open set $U\subset T$,
we can choose a local basis of $V_{m,p}$ with constant coefficients with respect to $E_{m}$,
i.e., 
the coefficients of the defining polynomials of $X$ in $\mathbb{P}(E_{m})$ over $U$ can
be choosen as constants.
Then the defining equations $V_{m,p}$ over $U\times s$ are the same as
$V_{m,p}$ over $U\times \{ 0\}$ for $s\in \Delta$.
Therefore $\mathcal{V}_{m,p}$ defines a smooth deformation of $X$, 
we denote it
$$\begin{CD} \mathcal{X} @>\pi >> \mathcal{T} @> \pi_{1}>> \Delta \end{CD}.$$
 
As for the second part of the proposition,
we first prove that $-K_{X_{t}}$ is ample on $X_{t}$ 
where $X_{t}$ is the fiber of $\mathcal{X}\rightarrow \mathcal{T}$ over $t\in \mathcal{T}$
and $t$ is in a neighborhood of $T$ in $\mathcal{T}$.
Let $t_0 \in T$. Let $U_{t_0}$ be a small neighborhood of $t_0$ in $\mathcal{T}$. 
Since $-K_{X_{t_{0}}}$ is ample, by \cite{Yau}
there exists a Kähler metric $\omega_{t_{0}}$ on $X_{t_{0}}$ such that $i\Theta_{\omega_{t_{0}}}(-K_{X_{t_{0}}})> 0$.
By a standard continuity argument (cf. for instance \cite[Theorem 3.1]{Sch}), if $U_{t_0}$ is small enough,
we can construct Kähler metrics $\omega_{t}$ on $X_{t}$ for any $t\in U_{t_0}$ 
and by continuity
the curvatures $i\Theta_{\omega_{t}}(-K_{X_{t}})$ are strictly positive.
Therefore $-K_{X_{t}}$ is ample on $X_{t}$ for any $t$ sufficient close to $t_{0}$ in $\mathcal{T}$.
Letting $t_{0}$ run over $T$,
then $-K_{X_{t}}$ is ample for all $t$ in a neighborhood of $T$ in $\mathcal{T}$.

We need also prove that $-K_{X_{s}}$ is nef on $X_{s}$,
where $X_{s}$ is the fiber of $\pi\circ\pi_{1}$ over $s\in \Delta$.
Let $(E_{m})_{s}$ be the fiber of $\mathcal{E}_{m}\rightarrow \Delta$ over $s$.
By construction, $(E_{m})_{s}$ is numerically flat on $T_{s}$, 
where $T_{s}$ is fiber of $\pi_{1}$ over $s$.
Then $\mathcal{O}_{\mathbb{P}(\mathcal{E}_{m})}(1)$ is nef on $\mathbb{P}(E_{m})_{s}$.
Since $X_{s}$ is embedded in $\mathbb{P}(E_{m})_{s}$,
$\mathcal{O}_{\mathbb{P}(\mathcal{E}_{m})}(1)|_{X_{s}}$
is also nef for any $s\in \Delta$.
If $s=0$, we have
$$\mathcal{O}_{\mathbb{P}(\mathcal{E}_{m})}(1)|_{X_{s}}=-mK_{X} .$$
Therefore
$$c_{1}(\mathcal{O}_{\mathbb{P}(\mathcal{E}_{m})}(1)|_{X_{s}})=c_{1}(-mK_{X_{s}})$$
for $s\in \Delta$ by the rigidity of integral classes.
Then the nefness of $\mathcal{O}_{\mathbb{P}(\mathcal{E}_{m})}(1)|_{X_{s}}$ implies that
$-mK_{X_{s}}$ is nef for all $s\in\Delta$.

The proposition is proved.

\end{proof}

\begin{remark}
In general, nefness is not an open condition
in families 
(cf.  \cite[Theorem 1.2.17]{Laz04}).
Thanks to the above construction, 
nefness is preserved under
deformation in our special case.
\end{remark}

Thanks to Proposition 4.2, we immediately get the following corollary.
\begin{cor}
Let $X$ be a compact Kähler manifold satisfying the condition in Proposition 4.2. 
Then $X$ can be approximated by projective varieties. Moreover, $\nd(-K_{X})=n-\dim T$. 
\end{cor}

\begin{proof}
We keep the notations in Proposition 4.2.
By Proposition 4.2, there exists a deformation of $X\rightarrow T$:
$$\begin{CD} \mathcal{X} @>\pi >> \mathcal{T} @> \pi_{1}>> \Delta \end{CD}$$ 
such that $\mathcal{T}\rightarrow \Delta$ is the local universal deformation of $T$
and $X\rightarrow T$ is the central fiber of this deformation.
By Proposition 3.1, there exists a sequence $s_{i}\rightarrow 0$ in $\Delta$ 
such that all $T_{s_{i}}$ are projective.
Using Proposition 4.2, we know that the fibers of
$$X_{s_{i}}\rightarrow T_{s_{i}}$$ 
are Fano manifolds. 
Then all $X_{s_{i}}$ are projective
and $X$ can be approximated by projective manifolds.

As for the second part of the corollary, by observing that $-K_{X}$ is
relatively ample, we have $\nd(-K_{X})\geq n-r$.
If $\nd(-K_{X})\geq n-r+1$, by the definition of numerical dimension we have
$$\int_{X}(-K_{X})^{n-r+1}\wedge\omega_{X}^{r-1}> 0 .$$
By continuity,
\begin{equation}\label{positiveneighbor}
 \int_{X_{s_{i}}}(-K_{X_{s_{i}}})^{n-r+1}\wedge\omega_{X_{s_{i}}}^{r-1}> 0
\end{equation}
for $|s_{i}|\ll 1$.
Thanks to Proposition 4.2, $-K_{X_{s_{i}}}$ are nef.
Then \eqref{positiveneighbor} implies the existence of a projective variety $X_{s_{i}}$ such that $-K_{X_{s_{i}}}$ is nef 
and $\nd(-K_{X_{s_{i}}})\geq n-r+1$,
which contradicts with the Kawamata-Viehweg vanishing theorem for projective varieties.
We get a contradiction and the corollary is proved.
\end{proof}

\section{A Kawamata-Viehweg vanishing theorem}

As pointed out in the introduction, 
when $X$ is a projective variety of dimension $n$
and $L$ is a nef line bundle on $X$ with $\nd (L)=k$,
we have the Kawamata-Viehweg vanishing theorem:
$$H^{r}(X, K_{X}+L)=0\qquad \text{for} \hspace{5 pt} r>n-k.$$
But it is probably a difficult problem to 
prove this vanishing theorem for a non projective 
compact Kähler manifold.
We will prove in this section a Kawamata-Viehweg vanishing theorem 
for certain Kähler manifolds. 
More precisely, we say that a compact Kähler manifold $X$ of dimension $n$ 
and a nef line bundle $L$ satisfy Conditions $(*)$, 
if

\textbf{Conditions} $(*)$:
Let $L$ be a nef line bundle on a  compact Kähler manifold $(X, \omega_X)$ of dimension $n$ . 
We say that $(X,L)$ satisfies Conditions $(*)$, if

$(i)$
There exists a smooth two steps tower fibration 
$$\begin{CD} X @>\pi >> T @> \pi_{1}>> S \end{CD}$$ 
where $\pi$ is a submersion to a smooth variety $T$ of dimension $r$,
and $\pi_{1}$ is a submersion to a smooth curve $S$.

$(ii)$
The nef line bundle $L$ is relatively ample with respect to $\pi$
and 
$$\pi_{*}(L^{n-r+1})=\pi_{1}^{*}(\mathcal{O}_{S}(1))$$ 
for some ample line bundle $\mathcal{O}_{S}(1)$ on $S$. 

\begin{remark}\label{hovanskiiinequality}
We first remark that $(ii)$ of Conditions $(*)$ implies that 
$$\nd(L)> n-r .$$
Using the Hovanskii-Teissier inequality for arbitrary compact K\"{a}hler manifolds \cite{Gro},
we obtain
\begin{equation}\label{interhovan}
\int_X L^{n-r+1}\wedge\omega_T ^{r-2}\wedge \pi^{*}\pi_{1}^{*}(\mathcal{O}_{S}(1))\geq
\end{equation}
$$ 
(\int_X L^{n-r+p}\wedge\omega_T ^{r-p-1}\wedge\pi^{*}\pi_{1}^{*}(\mathcal{O}_{S}(1)))^{\frac{1}{p}} 
(\int_X L^{n-r}\wedge\omega_T ^{r-1}\wedge\pi^{*}\pi_{1}^{*}(\mathcal{O}_{S}(1)))^{\frac{p-1}{p}} ,$$
where $\omega_T$ is a K\"{a}hler metric on $T$ and $p>1$.
Since $\dim S=1$, $(ii)$ in Conditions $(*)$ implies that 
$$\int_X L^{n-r+1}\wedge \pi^{*}\pi_{1}^{*}(\mathcal{O}_{S}(1))\wedge\omega_T ^{r-2}=0 .$$
Moreover, the relative ampleness of $L$ implies that 
$$\int_X L^{n-r}\wedge\omega_T ^{r-1}\wedge\pi^{*}\pi_{1}^{*}(\mathcal{O}_{S}(1)) > 0 .$$
By \eqref{interhovan}, we obtain
\begin{equation}\label{intercorhovanskii}
\int_X L^{n-r+p}\wedge\omega_T ^{r-p-1}\wedge \pi^{*}\pi_{1}^{*}(\mathcal{O}_{S}(1))=0 \qquad\text{for any }p\geq 1.
\end{equation}

Suppose that $\nd (L)=n-r+t$. If $t\geq 2$,
using again the Hovanskii-Teissier inequality,
we have
$$\int_X L^{n-r+1}\wedge\omega_X ^{r-2}\wedge \pi^{*}\pi_{1}^{*}(\mathcal{O}_{S}(1))\geq$$
$$ 
(\int_X L^{n-r+t}\wedge\omega_X ^{r-t-1}\wedge\pi^{*}\pi_{1}^{*}(\mathcal{O}_{S}(1)))^{\frac{1}{t}} 
(\int_X L^{n-r}\wedge\omega_X ^{r-1}\wedge\pi^{*}\pi_{1}^{*}(\mathcal{O}_{S}(1)))^{\frac{t-1}{t}} .$$
Since $L$ is relatively ample, $\omega_X$ is controled by $L+C\cdot \omega_T$ for some $C> 0$ large enough.
Then \eqref{intercorhovanskii} implies that 
$$\int_X L^{n-r+1}\wedge \pi^{*}\pi_{1}^{*}(\mathcal{O}_{S}(1))\wedge\omega_X ^{r-2}=0 .$$
Moreover, the relative ampleness of $L$ implies that 
$$\int_X L^{n-r}\wedge\omega_X ^{r-1}\wedge\pi^{*}\pi_{1}^{*}(\mathcal{O}_{S}(1)) > 0 .$$
we obtain finally
\begin{equation}\label{corhovanskii}
L^{n-r+t}\wedge \pi^{*}\pi_{1}^{*}(\mathcal{O}_{S}(1))=0 .
\end{equation}

\end{remark}

We will prove in this section that if $(X,L)$ satisfies Conditions $(*)$, 
then
$$H^{p}(X,K_{X}+L)=0 \qquad \text{for} \hspace{5 pt}p\geq r.$$
Before the proof of this vanishing theorem, we first prove a useful lemma.
\begin{lemme}
Assume that $(X,L)$ satisfies Conditions $(*)$.
Then $L-c\pi^{*}\pi_{1}^{*}(\mathcal{O}_{S}(1))$ is pseudo-effective for some 
constant $c>0$.
\end{lemme}

\begin{proof}
We first explain the idea of the proof.
By using a Monge-Ampère equation, 
we can construct a sequence of metrics $\{\varphi_{\epsilon}\}$ on $L$, such
that
$$\frac{i}{2\pi}\Theta_{\varphi_{\epsilon}}(L)\geq c\pi^{*}\pi_{1}^{*}(\mathcal{O}_{S}(1))
\qquad\text{for all small}\hspace{5 pt} \epsilon .$$
Then 
$\frac{i}{2\pi}\Theta_{\varphi}(L)\geq c\pi^{*}\pi_{1}^{*}\mathcal{O}_{S}(1)$,
where $\varphi$ is a limit of some subsequence of $\{\varphi_{\epsilon}\}$.
In this way, the lemma would therefore be proved.
This idea comes from \cite{DP04}, 
but the proof here is in some sense much simpler 
because we do not need their diagonal trick in our case.

Observing first that $(ii)$ of Conditions $(*)$ 
implies that $\nd (L)> n-r$,
we can thus suppose that $\nd (L)=n-r+t$, for some $t\geq 1$.
For simplicity, 
we denote $\pi^{*}\pi_{1}^{*}\mathcal{O}_{S}(1)$ by $A$.
Let $s\in S$, and $X_{s}$ the fiber of $\pi\circ\pi_{1}$ over $s$.
We first fix a smooth metric $h_{0}$ on $\mathcal{O}_{S}(1)$. 
Thanks to the semi-positivity of $A$,
we can choose a sequence of smooth functions $\psi_{\epsilon}$ on $X$
such that for the metrics $h_{0}e^{-\psi_{\epsilon}}$ on $A$,
the curvature forms $\frac{i}{2\pi}\Theta_{\psi_{\epsilon}}(A)$ are
semi-positive
\footnote{Note that here $\psi_{\epsilon}$ are functions, but the $\varphi$'s in
Proposition 2.1 are metrics! Therefore in this lemma,
$\frac{i}{2\pi}\Theta_{\psi_{\epsilon}}(\mathcal{O}_{S}(1))=\frac{i}{2\pi}\Theta_{h_{0}}(\mathcal{O}_{S}
(1))+dd^{c}\psi_{\epsilon}$.},
and 
\begin{equation}\label{equ11}
\int_{V_{\epsilon}}\frac{i}{2\pi}\Theta_{\psi_{\epsilon}}(A)\wedge \omega^{n-1}\geq C_{1}
 \qquad \text{for}\hspace{5 pt} \epsilon\rightarrow 0
\end{equation}
where $V_{\epsilon}$ is an $\epsilon$ open neighborhood of $X_{s}$, 
and $C_{1}>0$ is a uniform constant. 
(All the constants $C_{i}$ below will be uniformly strictly positive.
When the uniform boundedness comes from obvious reasons, we will not make it
explicit. )

Let $\tau_{1}, \tau_{2}$ two constants such that
$1\gg \tau_{1}\gg \tau_{2}>0$ which will be made precise later. 
Let $h$ be a fixed smooth metric on $L$.
Thanks to the nefness of $L$, 
we can solve a Monge-Ampère equation:
\begin{equation}\label{equ12}
(\frac{i}{2\pi}\Theta_{h}(L)+\tau_{1}\omega+dd^{c}\varphi_{\epsilon}
)^{n}=C_{2,\epsilon}\frac{\tau_{1}^{r-t}}{\tau_{2}^{n-1}}(\frac{i}{2\pi}\Theta_{\psi_{
\epsilon}}(A)+\tau_{2}\omega)^{n},
\end{equation}
where 
$$C_{2,\epsilon}=\frac{(\frac{i}{2\pi}\Theta_{h}(L)+\tau_{1}\omega)^{n}\tau_{2}^
{n-1}}{\tau_{1}^{r-t}(\frac{i}{2\pi}\Theta_{\psi_{
\epsilon}}(A)+\tau_{2}\omega)^{n} }.$$
Since $\nd(L)=n-r+t$ and $\dim S=1 $, 
we have $\inf\limits_{\epsilon} C_{2,\epsilon}> 0$.

Let $\lambda_{1}\leq\lambda_{2}\leq\cdots\leq\lambda_{n}$ be the eigenvalues of 
$\frac{i}{2\pi}\Theta_{h}(L)+\tau_{1}\omega+dd^{c}\varphi_{\epsilon}$ with respect to
$\frac{i}{2\pi}\Theta_{\psi_{\epsilon}}(A)+\tau_{2}\omega$.
Then the Monge-Ampère equation \eqref{equ12} implies that 
\begin{equation}\label{equ13}
\prod_{i=1}^{n}\lambda_{i}(z)=C_{2,\epsilon}\frac{\tau_{1}^{r-t}}{\tau_{2}^{n-1}} 
\qquad \text{for any} \hspace{5 pt} z\in X.
\end{equation}

We claim that there exists a constant $\delta>0$ independent of $\epsilon, \tau_{1}, \tau_{2}$,
such that
\begin{equation}\label{equ14}
 \int_{V_{\epsilon}\setminus
E_{\delta,\epsilon}}\frac{i}{2\pi}\Theta_{\psi_{\epsilon}}(A)\wedge\omega^{n-1}\geq
\frac{C_{1}}{2}\qquad\text{for any }\epsilon ,
\end{equation}
where 
$$E_{\delta,\epsilon}=\{z\in V_{\epsilon} | \prod_{i=2}^{n}\lambda_{i}(z)\geq
C_{2,\epsilon}\frac{\tau_{1}^{r-t}}{\delta\tau_{2}^{n-1}}\}.$$
We postphone the proof of the claim in Lemma 5.2 and finish the proof of this lemma.
Since
$$\lambda_{1}(z)\geq
\frac{C_{2}\frac{\tau_{1}^{r-t}}{\tau_{2}^{n-1}}}{C_{2}\frac{\tau_{1}^{r-t}}{
\delta\tau_{2}^{n-1}}}=\delta \qquad \text{for} \hspace{5 pt} 
z\in V_{\epsilon}\setminus E_{\delta,\epsilon} $$
by the definition of $E_{\delta,\epsilon}$ and \eqref{equ13},
\eqref{equ14} implies hence that
$$\int_{V_{\epsilon}}
(\frac{i}{2\pi}\Theta_{h}(L)+\tau_{1}\omega+dd^{c}\varphi_{\epsilon}
)\wedge\omega^{n-1}\geq C_{8}\int_{V_{\epsilon}} \lambda_{1}(z)
\frac{i}{2\pi}\Theta_{\psi_{\epsilon}}(A)\wedge\omega^{n-1}$$
\begin{equation}\label{equ15}
\geq\delta C_{8}\int_{V_{\epsilon}\setminus
E_{\delta,\epsilon}}\frac{i}{2\pi}\Theta_{\psi_{\epsilon}}(A)\wedge\omega^{n-1}\geq
\delta\cdot C_{8}\cdot\frac{C_{1}}{2}. 
\end{equation}
Letting $\epsilon\rightarrow 0$, 
the choice of $V_{\epsilon}$ and \eqref{equ15}
imply that the weak limit of 
$$\frac{i}{2\pi}\Theta_{h}(L)+\tau_{1}\omega+dd^{c}\varphi_{\epsilon}$$
is more positive than
$C_{9}[X_{s}]$.
Thus $L+\tau_{1}\omega-C_{9}[X_{s}]$ is pseudo-effective.
Since $C_{9}$ is independent of $\tau_{1}$, when $\tau_{1}\rightarrow 0$,
we obtain that $L-C_{9}\pi^{*}\pi_{1}^{*}(\mathcal{O}_{S}(1))$ is pseudo-effective.
The lemma is proved.

\end{proof}

\begin{remark}
The proof of the lemma 5.1 does not use the smoothness of $\pi$. 
Lemma 5.1 is also true under the weaker assumption that
the generic fiber of $\pi$ is irreducible.
\end{remark}

\begin{lemme}
We now prove the claim in Lemma 5.1
\end{lemme}

\begin{proof}
By construction, 
\begin{equation}\label{equ21}
\int_{X}(\prod_{i=2}^{n}\lambda_{i}(z))(\frac{i}{2\pi}\Theta_{\psi_{\epsilon}}(A)+\tau_{2}\omega)^{n}
\end{equation}
$$\leq
C_{3}\int_{X}(c_{1}(L)+\tau_{1}\omega+dd^{c}\varphi_{
\epsilon})^{n-1}\wedge (\frac{i}{2\pi}\Theta_{\psi_{\epsilon}}(A)+\tau_{2}\omega)$$
$$=C_{3} \int_{X}(c_{1}(L)+\tau_{1}\omega)^{n-1}\wedge
(c_{1}(A)+\tau_{2}\omega).$$
On the other hand, using \eqref{corhovanskii} in Remark \ref{hovanskiiinequality} after Conditions $(*)$, 
we have
\begin{equation}\label{equ22}
\int_{X}(c_{1}(L)+\tau_{1}\omega)^{n-1}\wedge (c_{1}(A)+\tau_{2}\omega)
\end{equation}
$$=C_{4}\tau_{1}^{r-t}c_{1}(L)^{n-r+t-1}\wedge c_{1}(A)+O(\tau_{2})\leq C_{5}\tau_{1}^{r-t}.$$
where the last inequality comes from the fact that $\tau_{2}\ll \tau_{1}$.
Combining \eqref{equ21} and \eqref{equ22}, we have 
\begin{equation}\label{equ23}
\int_{X}(\prod_{i=2}^{n}\lambda_{i}(z))(\frac{i}{2\pi}\Theta_{\psi_{\epsilon}}(A)+\tau_{2}
\omega)^{n}\leq C_{6}\tau_{1}^{r-t}.
\end{equation}

For any $\delta$ fixed,
\eqref{equ23} and the definition of $E_{\delta,\epsilon}$ imply that
$$\int_{E_{\delta}}C_{2,\epsilon}\frac{\tau_{1}^{r-t}}{\delta\tau_{2}^{n-1}}
(\frac{i}{2\pi}\Theta_{\psi_{\epsilon}}(A)+\tau_{2}\omega)^{n}\leq C_{6}\tau_{1}^{r-t}.$$
Combining with the fact that $\inf\limits_{\epsilon} C_{2,\epsilon} > 0$, we get
\begin{equation}\label{equ24}
 \int_{E_{\delta,\epsilon}}(\frac{i}{2\pi}\Theta_{\psi_{\epsilon}}(A)+\tau_{2}\omega)^{n
}\leq C_{7}\delta\tau_{2}^{n-1}.
\end{equation}
Since $\frac{i}{2\pi}\Theta_{\psi_{\epsilon}}(A)$ is semi-positive, 
\eqref{equ24} implies that
\begin{equation}\label{equ25}
\int_{E_{\delta,\epsilon}} \frac{i}{2\pi}\Theta_{\psi_{\epsilon}}(A)\wedge\omega^{n-1}\leq C_{7}\delta.
\end{equation}
By taking $\delta=\frac{C_{1}}{2C_{7}}$, \eqref{equ11} of Lemma 5.1 and \eqref{equ25} imply that
$$\int_{V_{\epsilon}\setminus
E_{\delta,\epsilon}}\frac{i}{2\pi}\Theta_{\psi_{\epsilon}}(A)\wedge\omega^{n-1}\geq
\frac{C_{1}}{2}.$$
The lemma is proved.
\end{proof}

Using Lemma 5.1, we would like to prove a Kawamata-Viehweg type vanishing theorem.
Recall that T.Ohsawa proved in \cite{Ohs}: 
if $X\rightarrow T$ is a smooth fibration and
$(E, h)$ is a hermitian line bundle on $X$ with $\frac{i}{2\pi}\Theta_{h}(E)\geq \pi^{*}\omega_{T}$.
Then 
$$H^{q}(T, R^{0}\pi_{*}(K_{X}\otimes E))=0$$
for $q\geq 1$.
In his proof, he uses the metrics $\pi^{*}\omega_{T}+\tau\omega_{X}$ on $X$
and lets
$\tau\rightarrow 0$  to preserve the information on $T$.
The idea of our proof comes from this technique.

\begin{prop}
Assume that $(X, L)$ satisfies Conditions $(*)$. Then
$$H^{p}(X,K_{X}+L)=0 \qquad \text{for } p\geq r.$$
\end{prop}

\begin{proof}
Thanks to Conditions $(*)$, we have a smooth fibration
$$\begin{CD} X @>\pi >> T @> \pi_{1}>> S \end{CD} .$$
Using the fixed metric $\omega_{X}$, we have a $C^{\infty}$-decomposition of the tangent bundle of $X$:
$$T_{X}=T_{X/T}\oplus E_{1} \oplus E_{2},$$
where $T_{X/T}$ is the relative tangent bundle of $\pi: X\rightarrow T$, 
$E_{1}$ is the pull back of the relative tangent bundle of $\pi_{1}: T\rightarrow S$
and $E_{2}$ is the pull back of the tangent bundle of $S$.

We first construct a metric $h$ with analytic singularities on $L$ such that
$i\Theta_{h}(L)$ satisfies the following three conditions.

$(I)$. $i\Theta_{h}(L)$ is strictly positive on $T_{X/T}$.

$(II)$. The restrictions of $i\Theta_{h}(L)$ on $E_{1}$ maybe negative, 
but the positivity of the restrictions on $E_{2}$ is large enough with respect to the negativity on $E_{1}$.  

$(III)$. $\mathcal{I}(h)=\mathcal{O}_{X}$.

\hspace{-12pt}By Lemma 5.1 and Demailly's regularization theorem,
for any $\epsilon> 0$,
there is a singular metric $h_{1}$ with analytic singularities such that 
\begin{equation}\label{equ31}
 i\Theta_{h_{1}}(L)\geq c\pi^{*}(\omega_{S})-\epsilon\omega_{X} .
\end{equation}
We will make the choice of $\epsilon$ more explicit later on.
Moreover, since $L$ is relatively ample, there is a smooth metric $h_{2}$ on $L$,
such that the restriction of $i\Theta_{h_{2}}(L)$ on $T_{X/T}$ is strictly positive.
Thanks to the nefness of $L$, we can also choose a smooth metric $h_{\epsilon}$ on $L$ 
such that 
\begin{equation}\label{equ32}
i\Theta_{h_{\epsilon}}(L)\geq -\epsilon\omega_{X}.
\end{equation}
We now define a new metric $h$ on $L$ by combining the above three metrics:
$$h=\epsilon_{1} h_{1}+\epsilon_{2} h_{2}+ (1-\epsilon_{1}-\epsilon_{2})h_{\epsilon}$$
for some $1\gg \epsilon_{1}\gg \epsilon_{2}\gg \epsilon> 0$.

We now check that the new metric $h$ satisfies the above three conditions.
Since $1\gg \epsilon_{1}\gg\epsilon_{2}$ and $h_{\epsilon}$ is smooth, 
Condition $(III)$ follows.
To check the first two properties, we first observe that \eqref{equ31} and \eqref{equ32} imply that
$$i\Theta_{h}(L)\geq c\epsilon_{1} \pi^{*}(\omega_{S})+\epsilon_{2} i\Theta_{h_{2}}(L)-\epsilon\omega_{X}.$$
Therefore it is enough to check the conditions $(I)$ and $(II)$ for 
$c\epsilon_{1} \pi^{*}(\omega_{S})+\epsilon_{2} i\Theta_{h_{2}}(L)-\epsilon\omega_{X}$.
Condition $(I)$ follows by the observation that 
$i\Theta_{h_{2}}$ is strictly positive on $T_{X/T}$ and $\epsilon_{2}\gg \epsilon$.
Since $\epsilon_{1}\gg \epsilon_{2}$,
the positivity of $i\Theta_{h}(L)$ on the direction of $E_{2}$ 
is large enough with respect to the negativity on the directions of $E_{1}$,
which comes from $\epsilon_{2} i\Theta_{h_{2}}(L)$.
Condition $(II)$ follows.

Let $\omega_{T}$ be a Kähler metric on $T$ and let
$\omega_{\tau}=\tau\omega_{X}+\pi^{*}(\omega_{T})$ for $\tau> 0$.
When $\tau\rightarrow 0$, Condition $(I)$ and Condition $(II)$ of $h$ imply that
the pair $(X, \omega_{\tau}, L, h)$
satisfies \eqref{strictpositive} in Proposition 2.4.
Then
$$H^{p}(X,K_{X}\otimes L\otimes\mathcal{I} (h))=0
\qquad \text{for } p\geq r.$$
Since $\mathcal{I}(h)=\mathcal{O}_{X}$ by our construction,
we get
$$H^{p}(X,K_{X}\otimes L)=0\qquad \text{for } p\geq r.$$

\end{proof}

\section{Deformation of compact Kähler manifolds with nef tangent bundles}

Before giving the proof of our main theorem, we need a technical lemma.

\begin{lemme}
Assume that $X$ has a two step tower smooth fibration:
$$\begin{CD} X @>\pi >> T @> \pi_{1}>> S \end{CD} ,$$ 
where $T$ is a torus of dimension $r$, and $S$ is an abelian variety of dimension $s$.
We suppose also that the fibers of $\pi$ are Fano manifolds.
Let $S_{p}$ be a complete intersection of zero divisors of 
$p$ general global sections of a very ample line bundle $\mathcal{O}_{S}(1)$ on $S$.
Let $T_{p}$ and $X_{p}$ be the inverse images of $S_{p}$ in $T$ and $X$.
Then 
$$H^{r-p}(X_{p},K_{X_{p}}-K_{X})\neq 0\qquad \text{for} \hspace{5 pt} 0 \leq p\leq s-1.$$
\end{lemme}

\begin{proof}
By the adjunction formula 
$-K_{X_{p}}+p\pi^{*}\pi_{1}^{*}\mathcal{O}_{S}(1)=-K_{X}|_{X_{p}} ,$
we have
\begin{equation}\label{equation11}
H^{r-p}(X_{p},K_{X_{p}}-K_{X})=H^{r-p}(X_{p}, p\pi^{*}\pi_{1}^{*}\mathcal{O}_{S}(1) ).
\end{equation}
On the other hand,  
the assumption that the fibers of $\pi$ are Fano manifolds implies that
\begin{equation}\label{equation12}
H^{r-p}(X_{p}, p\pi^{*}\pi_{1}^{*}\mathcal{O}_{S}(1) )=H^{r-p}(T_{p}, p\pi_{1}^{*}\mathcal{O}_{S}(1))
\end{equation}
by using the Leray spectral sequence.
Observing moreover that $K_{T_{p}}=p\pi_{1}^{*}\mathcal{O}_{S}(1)$,
\eqref{equation11} and \eqref{equation12} imply that
\begin{equation}\label{equation13}
 H^{r-p}(X_{p},K_{X_{p}}-K_{X})=H^{r-p}(T_{p}, K_{T_{p}})
\qquad \text{for} \hspace{5 pt} 0 \leq p\leq s-1 .
\end{equation}
To prove the lemma, 
it is therefore enough to check the non vanishing of $H^{r-p}(T_{p},
K_{T_{p}})$.

Since $\dim T_{p}-\dim S_{p}=r-s$ for any $p$, 
by a standard vanishing theorem (cf.
\cite[Theorem 4.1, Chapter VII]{Dem}), we have
\begin{equation}\label{equation14}
 H^{r-s+i}(T_{p}, K_{T_{p}} +\pi_{1}^{*}\mathcal{O}_{S}(1))=0
\end{equation}
for $i\geq 1$ and $p=0,1,...,s-1$.
Thanks to the exact sequence
$$0\rightarrow \mathcal{O}_{T_{p-1}}( K_{T_{p-1}})\rightarrow
\mathcal{O}_{T_{p-1}}( K_{T_{p-1}}+\pi_{1}^{*}\mathcal{O}_{S}(1))
\rightarrow \mathcal{O}_{T_{p}} ( K_{T_{p}})\rightarrow 0 ,$$
\eqref{equation14} implies that
\begin{equation}\label{equation15}
 H^{r-s+i}(T_{p}, K_{T_{p}})=H^{r-s+i+1}(T_{p-1}, K_{T_{p-1}}) \qquad \text{for} \hspace{5 pt} i \geq 1.
\end{equation}
In particular, if we take $i=s-p$ in \eqref{equation15}, 
then
\begin{equation}\label{equation16}
H^{r-p}(T_{p}, K_{T_{p}})=H^{r-p+1}(T_{p-1}, K_{T_{p-1}}) .
\end{equation}
Since $T_{0}=T$ is a torus,  we have $H^{r}(T_{0}, K_{T_{0}})\neq 0$.
Then \eqref{equation16} implies by induction that 
$$H^{r-p}(T_{p}, K_{T_{p}})\neq 0 \qquad \text{for} \hspace{5 pt} 0 \leq p\leq s-1.$$
Using \eqref{equation13}, the lemma is proved. 
\end{proof}

\begin{thm}
Let $X$ be a compact Kähler manifold of dimension $n$ with nef tangent bundle, 
and $\pi: X\rightarrow T$ a smooth fibration onto a torus $T$ of dimension $r$ 
such that $-K_{X}$ is nef and relatively ample. 
Then $\nd (-K_{X})=n-r$.
\end{thm}

\begin{proof}

We prove the theorem by contradiction.
Observing that the relative ampleness of $-K_{X}$ already implies that
$\nd(-K_{X})\geq n-r$, 
we can thus assume by contradiction that $\nd(-K_{X})\geq n-r+1$. 
There are two cases.

\textbf{Case 1:} $\pi_{*}((-K_{X})^{n-r+1})$ is trivial on $T$. 

Then
\begin{equation}\label{equation21}
\int_{X}c_{1}(-K_{X})^{n-r+1}\wedge(\pi^{*}\omega_{T})^{r-1}=0 ,
\end{equation}
where $\omega_{T}$ is a Kähler form on $T$.
Since $T_{X}$ is nef, by \cite[Corollary 2.6]{DPS94}, \eqref{equation21} implies that
all degree $n-r+1$ Chern polynomals $P$ of $T_{X}$ satisfy
\begin{equation}\label{equation22}
\int_{X}P(T_{X})\wedge(\pi^{*}\omega_{T})^{r-1}=0 .
\end{equation}
Let $E_{m}=\pi_{*}(-mK_{X})$.
By the Riemann-Roch-Grothendick theorem, we have
\begin{equation}\label{equation23}
\Ch(E_{m})=\pi_{*}(c_{1}(-K_{X})\cdot \Todd(T_{X})).
\end{equation}
Then \eqref{equation22} and \eqref{equation23} imply that $\int_T c_{1}(E_{m})\wedge(\omega_{T})^{r-1}=0$. 
But by Proposition 2.1, $E_{m}$ is nef on $T$. 
Therefore $E_{m}$ is numerically flat.
By Corollary 4.3, $\nd(-K_{X})=n-r$. 
We get a contradiction.

\textbf{Case 2:} $\pi_{*}((-K_{X})^{n-r+1})$ is a non trivial class on $T$. 

Since $-K_X$ is also nef, we have
\begin{equation} \label{eqna}
\int_{T}\pi_{*}(c_{1}(-K_X)^{n-r+1})\wedge\omega_{T}^{r-1}> 0
\end{equation}
for any K\"ahler class $\omega_{T}$.
Using the assumption that $T$ is a torus, 
we can represent the cohomology class $\pi_{*}(c_{1}(-K_X)^{n-r+1})$ by
a constant $(1,1)$-form 
$$\sum_{i=1}^{r}\lambda_{i}d z_{i}\wedge d\overline{z}_{i}$$
on $T$.
Since \eqref{eqna} is valid for any K\"ahler class $\omega_{T}$, an elementary computation shows that 
$\lambda_{i}\geq 0$ for any $i$. Thus 
$\pi_{*}(c_{1}(-K_X)^{n-r+1})$ is a semipositive 
(non trivial) class in $H^{1,1} (T) \cap H^2(T, \mathbb{Q})$.
By Proposition 2.2,  we have the following smooth fibration
$$\begin{CD} X @>\pi >> T @> \pi_{1}>> S \end{CD} ,$$
where $S$ is an abelian variety of dimension $m$, and 
\begin{equation}\label{equation24}
\pi_{*}((-K_{X})^{n-r+1})=c\cdot \pi_{1}^{*}\mathcal{O}_{S}(1)
\end{equation}
for some $c>0$ and a very ample line bundle $\mathcal{O}_{S}(1)$ on $S$. 
Let $S_{m-1}$ be a complete intersection of $m-1$ general global sections of $H^{0}(S,\mathcal{O}_{S}(1))$.
Let $X_{m-1}$ and $T_{m-1}$ be the inverse images of $S_{m-1}$ in $X$ and $T$.
Then we have a smooth fibration
\begin{equation}\label{equation25}
\begin{CD} X_{m-1} @>\widetilde{\pi} >> T_{m-1} @> \widetilde{\pi}_{1}>> S_{m-1} \end{CD} .
\end{equation}
By \eqref{equation24}, we obtain that $(X_{m-1}, -K_{X}|_{X_{m-1}})$ and \eqref{equation25}
satisfy Conditions $(*)$ in Section 5.

Since $\dim T_{m-1}=r-(m-1)$, applying Proposition 5.3 to $(X_{m-1}, -K_{X}|_{X_{m-1}})$,
we obtain
$$H^{r-(m-1)}(X_{m-1}, K_{X_{m-1}}-K_{X})=0.$$
On the other hand, Lemma 6.1 implies that
$$H^{r-(m-1)}(X_{m-1}, K_{X_{m-1}}-K_{X})\neq 0.$$
We obtain again a contradiction.

Since Case 1 and Case 2 are both impossible, 
we infer that $\nd (-K_{X})=n-r$.

\end{proof}

Now we can prove our main result:
\begin{thm}
Let $X$ be a compact Kähler manifold of dimension $n$ with nef tangent bundle. 
Then $X$ can be approximated by projective varieties.
\end{thm}

\begin{proof}
By Lemma 2.5, there exists a finite étale Galois cover $\widetilde{X}\rightarrow X$ with group $G$ 
such that one has a commutative diagram
$$\begin{CD}
  \widetilde{X} @>>> X\\
 @VV\widetilde{\pi} V @VV\pi V\\
  T @>>> T/ G
  \end{CD}
$$
where the fibers of $\pi$ are Fano manifolds.
We suppose that $\dim T=r$.
Thanks to Theorem 6.2, we have $\nd(-K_{\widetilde{X}})=n-r$, 
which is equivalent to say that
$\nd (-K_{X})=n-r.$

Let $E_{m}=\pi_{*}(-mK_{X})$, for $m\geq 1$. 
Since $K_{T/G}$ is flat, by Proposition 2.1, 
$E_{m}$ is a nef vector bundle.
By the Riemann-Roch-Grothendick theorem, we have
$$\Ch(E_{m})=\pi_{*}(\Ch(-K_{X})\Todd (T_{X})) .$$
Since  $\nd (-K_{X})=n-r$,  the above equality implies that
$c_{1}(E_{m})=0$ 
by using \cite[Corollary 2.6]{DPS94}.
Then $E_{m}$ is numerically flat.
Using Corollary 4.3, we conclude that $X$ can be approximated by projective varieties.

\end{proof}

\bibliographystyle{alpha}
\bibliography{biblio}

\end{document}